\documentclass[12pt,a4paper]{amsart}
\usepackage{mathtools}
\usepackage{xcolor}
\usepackage{amsmath,amstext,amssymb,amsthm,amsfonts}
\usepackage[all]{xy} \xyoption{poly}
\newcommand{\nc}{\newcommand}

\nc{\one}{\mbox{\bf 1}}
\nc{\invtensor}{\underset{\leftarrow}{\otimes}}
\nc{\const}{\operatorname{const}}

\nc{\ad}{\operatorname{ad}}

\nc{\tr}{\operatorname{tr}}

\nc{\tp}{\operatorname{top}}
\nc{\rank}{\operatorname{rank}}
\nc{\corank}{\operatorname{corank}}
\nc{\codim}{\operatorname{codim}}
\nc{\sdim}{\operatorname{sdim}}
\nc{\mult}{\operatorname{mult}}
\nc{\ds}{\operatorname{ds}}
\nc{\tail}{\operatorname{tail}}
\nc{\howl}{\operatorname{howl}}
\nc{\spn}{\operatorname{span}}
\nc{\defect}{\operatorname{defect}}
\nc{\Sym}{\operatorname{Sym}}
\nc{\sym}{\operatorname{sym}}
\nc{\id}{\operatorname{id}}
\nc{\Id}{\operatorname{Id}}
\nc{\Ree}{\operatorname{Re}}
\nc{\hi}{\operatorname{hi}}
\nc{\htt}{\operatorname{ht}}
\nc{\at}{\operatorname{at}}
\nc{\str}{\operatorname{str}}
\nc{\Iso}{\operatorname{Iso}}
\nc{\Ker}{\operatorname{Ker}}
\nc{\rker}{\operatorname{rKer}}
\nc{\im}{\operatorname{Im}}
\nc{\osp}{\mathfrak{osp}}
\nc{\sgn}{\operatorname{sgn}}
\nc{\F}{\operatorname{F}}
\nc{\Mod}{\operatorname{Mod}}
\nc{\DS}{\operatorname{DS}}
\nc{\Soc}{\operatorname{Soc}}
\nc{\Sch}{\operatorname{Sch}}
\nc{\Hom}{\operatorname{Hom}}
\nc{\End}{\operatorname{End}}
\nc{\supp}{\operatorname{supp}}
\nc{\Card}{\operatorname{Card}}
\nc{\Ann}{\operatorname{Ann}}
\nc{\Ind}{\operatorname{Ind}}
\nc{\Coind}{\operatorname{Coind}}
\nc{\wt}{\operatorname{hwt}}
\nc{\hwt}{\operatorname{wt}}
\nc{\arc}{\operatorname{arc}}
\nc{\Arc}{\operatorname{Arc}}
\nc{\ch}{\operatorname{ch}}
\nc{\sch}{\operatorname{sch}}
\nc{\mdim}{\operatorname{mdim}}
\nc{\Stab}{\operatorname{Stab}}
\nc{\Irr}{\operatorname{Irr}}
\nc{\Spec}{\operatorname{Spec}}
\nc{\Res}{\operatorname{Res}}
\nc{\Aut}{\operatorname{Aut}}
\nc{\Ext}{\operatorname{Ext}}
\nc{\Prec}{\operatorname{Prec}}
\nc{\Fract}{\operatorname{Fract}}
\nc{\gr}{\operatorname{gr}}
\nc{\deff}{\operatorname{def}}
\nc{\core}{\operatorname{core}}
\nc{\HC}{\operatorname{HC}}
\nc{\dpth}{\operatorname{dpth}}
\nc{\sw}{\operatorname{sw}}
\nc{\red}{\operatorname{red}}

\nc{\pos}{\operatorname{pos}}

\nc{\wdchi}{\widetilde{\chi}}
\nc{\wdH}{\widetilde{H}}
\nc{\wdN}{\widetilde{N}}
\nc{\wdM}{\widetilde{M}}
\nc{\wdO}{\widetilde{O}}
\nc{\wdR}{\widetilde{R}}

\nc{\wdV}{\widetilde{V}}

\nc{\wdC}{\widetilde{C}}

\nc{\pari}{\operatorname{dex}}
\nc{\atyp}{\operatorname{atyp}}

\nc{\Core}{\operatorname{Core}}

\nc{\Obj}{\operatorname{Obj}}
\nc{\Dglie}{\operatorname{{\mathcal D}glie}}
\nc{\Fin}{\operatorname{{\mathcal F}in}}
\nc{\pr}{\operatorname{pr}}
\nc{\Adm}{\operatorname{\mathcal{A}dm}}

\nc{\Sg}{{\cS(\fg)}}
\nc{\Shg}{{\cS(\fhg)}}
\nc{\Ug}{{\cU(\fg)}}
\nc{\Uhg}{{\cU(\fhg)}}
\nc{\Sh}{{\cS(\fh)}}
\nc{\Uh}{{\cU(\fh)}}
\nc{\Uhh}{{\cU(\fhh)}}
\nc{\Zg}{{{\mathcal{Z}}(\fg)}}

\nc{\Vir}{{\mathcal{V}ir}}
\nc{\NS}{{\mathcal{N}S}}

\nc{\tZg}{{\widetilde{\mathcal Z}({\mathfrak g})}}
\nc{\Zk}{{\mathcal Z}({\mathfrak k})}

\nc{\Up}{{\mathcal U}({\mathfrak p})}
\nc{\Ah}{{\mathcal A}({\mathfrak h})}
\nc{\Ag}{{\mathcal A}({\mathfrak g})}
\nc{\Ap}{{\mathcal A}({\mathfrak p})}
\nc{\Zp}{{\mathcal Z}({\mathfrak p})}
\nc{\cR}{\mathcal R}
\nc{\cS}{\mathcal S}
\nc{\cP}{\mathcal P}
\nc{\cT}{\mathcal{T}}
\nc{\CC}{\mathcal C}
\nc{\cA}{\mathcal A}
\nc{\cE}{\mathcal E}
\nc{\cU}{\mathcal U}
\nc{\cZ}{\mathcal Z}
\nc{\cN}{\mathcal N}
\nc{\cM}{\mathcal M}
\nc{\cL}{\mathcal L}
\nc{\cF}{\mathcal F}
\nc{\fg}{\mathfrak g}
\nc{\cB}{\mathcal{B}}

\nc{\fo}{\mathfrak o}

\nc{\CO}{\mathcal O}
\nc{\CR}{\mathcal R}

\nc{\cK}{\mathcal{K}}
\nc{\cW}{\mathcal{W}}
\nc{\bM}{\mathbf{M}}
\nc{\bL}{\mathbf{L}}
\nc{\bN}{\mathbf{N}}

\nc{\zq}{\mathpzc q}

\nc{\fl}{\mathfrak l}
\nc{\fn}{\mathfrak n}
\nc{\fm}{\mathfrak m}
\nc{\fp}{\mathfrak p}
\nc{\fh}{\mathfrak h}
\nc{\ft}{\mathfrak t}
\nc{\fk}{\mathfrak k}
\nc{\fb}{\mathfrak b}
\nc{\fs}{\mathfrak s}
\nc{\fr}{\mathfrak r}
\nc{\fB}{\mathfrak B}

\nc{\vareps}{\varepsilon}
\nc{\varesp}{\varepsilon}
\nc{\veps}{\varepsilon}

\nc{\fsl}{\mathfrak{sl}}
\nc{\fgl}{\mathfrak{gl}}
\nc{\fso}{\mathfrak{so}}
\nc{\fosp}{\mathfrak{osp}}
\nc{\fsp}{\mathfrak{sp}}
\nc{\fq}{\mathfrak q}
\nc{\fsq}{\mathfrak{sq}}
\nc{\fpsq}{\mathfrak{psq}}

\nc{\fhg}{\hat{\fg}}
\nc{\fhn}{\hat{\fn}}
\nc{\fhh}{\hat{\fh}}
\nc{\fhb}{\hat{\fb}}
\nc{\hrho}{\hat{\rho}}

\nc{\hsl}{\hat{\fsl}}
\nc{\fpo}{\mathfrak{po}}
\nc{\dirlim}{\underset{\rightarrow}{\lim}\,}
\nc{\nen}{\newenvironment}
\nc{\ol}{\overline}
\nc{\ul}{\underline}
\nc{\ra}{\rightarrow}
\nc{\lra}{\longrightarrow}
\nc{\Lra}{\Longrightarrow}
\nc{\bo}{\bar{1}}
\nc{\Lla}{\Longleftarrow}

\nc{\Llra}{\Longleftrightarrow}

\nc{\thla}{\twoheadleftarrow}

\nc{\lang}{(}
\nc{\rang}{)}

\nc{\hra}{\hookrightarrow}

\nc{\iso}{\overset{\sim}{\lra}}

\nc{\ssubset}{\underset{\not=}{\subset}}

\nc{\vac}{|0\rangle}

\nc{\simka}{{\ \scriptscriptstyle _{{\sim}}^\text{\tiny{k}}\ }}

\nc{\Thm}[1]{Theorem~\ref{#1}}
\nc{\Prop}[1]{Proposition~\ref{#1}}
\nc{\Lem}[1]{Lemma~\ref{#1}}
\nc{\Cor}[1]{Corollary~\ref{#1}}
\nc{\Conj}[1]{Conjecture~\ref{#1}}
\nc{\Claim}[1]{Claim~\ref{#1}}
\nc{\Defn}[1]{Definition~\ref{#1}}
\nc{\Exa}[1]{Example~\ref{#1}}
\nc{\Rem}[1]{Remark~\ref{#1}}
\nc{\Note}[1]{Note~\ref{#1}}
\nc{\Quest}[1]{Question~\ref{#1}}
\nc{\Hyp}[1]{Hypoth\`ese~\ref{#1}}
\nen{thm}[1]{\label{#1}{\bf Theorem.\ } \em}{}
\nen{prop}[1]{\label{#1}{\bf Proposition.\ } \em}{}
\nen{lem}[1]{\label{#1}{\bf Lemma.\ } \em}{}
\nen{cor}[1]{\label{#1}{\bf Corollary.\ } \em}{}
\nen{conj}[1]{\label{#1}{\bf Conjecture.\ } \em}{}

\nen{claim}[1]{\label{#1}{\bf Claim.\ } \em}{}

\nen{defn}[1]{\label{#1}{\bf Definition.\ } }{}
\nen{exa}[1]{\label{#1}{\bf Example.\ } }{}
\nen{rem}[1]{\label{#1}{\em Remark.\ } }{}
\nen{note}[1]{\label{#1}{\em Note.\ } }{}
\nen{exer}[1]{\label{#1}{\em Exercise.\ } }{}
\nen{sket}[1]{\label{#1}{\em Sketch of proof.\ } }{}
\nen{quest}[1]{\label{#1}{\bf Question.\ } \em}{}

\nen{hyp}[1]{\label{#1}{\bf Hypoth\`ese.\ } \em}{}
\begin{document}
\setcounter{section}{-1}
\setcounter{tocdepth}{1}

\title[ ]
{Bipartite extension graphs and the Duflo--Serganova functor}

\author{Maria Gorelik }

\address[]{Dept. of Mathematics, The Weizmann Institute of Science,Rehovot 76100, Israel}
\email{maria.gorelik@weizmann.ac.il}





\begin{abstract}
We consider several examples when the extension graph  admits  a bipartition compatible with
the action of Duflo--Serganova functors.
\end{abstract}

\maketitle

\section{Introduction}
In this paper $\fg=\fg_0\oplus \fg_1$ is  a complex
Lie superalgebra and  $x\in\fg_1$ is such that $[x,x]=0$.
This text is based on some results and ideas of~\cite{MS}, \cite{GS}  and~\cite{Skw}.  
We give a more detailed presentation of some parts of~\cite{MS}, \cite{GS} which 
are slightly different for $\osp$ and $\fgl$-cases.

This text is not intended for publication; its different parts 
have been developed in other papers~\cite{GH2} and~\cite{Gbi}.

\subsection{}\label{introDex}
Let $\CC$ be a category of representations of
a Lie superalgebra $\fg$ and $\Irr(\CC)$ be the set of isomorphism classes
of simple modules in $\CC$.  Assume that the modules in $\CC$ are of  finite length.
\footnote{ this  can be substituted by existence of local composition series  constructed in~\cite{DGK}.}
In many examples the extension graph of $\CC$ is bipartite, i.e. there exists a map
$\pari:\Irr(\CC)\to \mathbb{Z}_2$ such that

\begin{enumerate}
\item[$ \,$]
(Dex1) $\Ext^1_{\CC}(L_1,L_2)=0\ $ if
$\ \ \pari(L_1)=\pari(L_2)$.
\end{enumerate}
In this paper we are interested in  examples when the map $\pari$ is ``compatible''  with the  Duflo-Serganova functors $\DS_x$, i.e. 
\begin{enumerate}
\item[$ \,$]
(Dex2)  one has $\ [\DS_x(L):L']=0\ $ if $\ \pari(L)\not=\pari(L')$.
\end{enumerate}

Note that in  (Dex2) we have to choose $\pari$ on $\Irr(\CC)$ and on 
$\Irr(\DS_x(\CC))$. If   $\pari$  satisfies (Dex1) and  (Dex2), 
 then $\DS_x(L)$ 
is completely reducible for each $L\in\CC$.

We say that a module $M$ is {\em pure} if for any subquotient
$L$ of $M$, $\Pi(L)$ is not a subquotient of $M$. Note that
(Dex2) implies the purity of $\DS_x(L)$ for each $L\in\Irr(\CC)$.

In what follows $\Fin(\fg)$ stands for the full subcategory 
of finite-dimensional $\fg$-modules which are competely reducible over
 $\fg_0$.
In this paper we consider the case when $\fg$ is a finite-dimensional 
Kac-Moody superalgebra and $\CC=\Fin(\fg)$; for this case
we require that (Dex2) holds for each $x\in X(\fg)$, where
$$X(\fg):=\{x\in\fg_1|\ ][x,x]=0\}.$$

For other cases it make sense to restrict (Dex2)
to certain values of $x$: for instance, for affine superalgebras $\fg$  it makes sense to consider $x$ such that $\DS_x(\fg)$
is affine  and for
$\CC=\CO(\fg)$ it makes sense to consider $x$ "preserving"  the category $\CO$
(see~\cite{Gdepth}, Sections 7 and 8 respectively).

The examples with $\pari$ satisfying (Dex1) and (Dex2) include
$\Fin(\fg)$ for $\fg=\fgl(m|n)$
(this follows from~\cite{HW}) and the full subcategory 
of integrable modules  in the category $\CO(\fgl(1|n)^{(1)})$
(this follows from~\cite{GSaff}). 
In Section~\ref{sectatyp1}
we will check the existence of $\pari$  for   
exceptional Lie superalgebras. 
We will introduce $\pari$ satisfying (Dex1) for $\osp(m|n)$-case;
in~\cite{GH} we will show that (Dex2) holds in this case too.
It would be interesting to find other examples
 with $\pari$ satisfying (Dex1)  and (Dex2). 
The strange superalgebras $\fp_n,\fq_n$ do not
admit $\pari$ satisfying (Dex1), (Dex2). 
By~\cite{GNS}, for $n\leq 3$,
 $\Fin(\fq_n)$ admits $\pari$ satisfying (Dex1);
this does not hold for larger $n$, see~\cite{MM}.
Moreover, the module $\DS_x(L)$ is not pure 
for each atypical $L\in\Irr(\Fin(\fq_2))$
(and $x\not=0$), see~\cite{Gdepth}, 5.5.2.
By contrast, the module $\DS_x(L)$ is  pure 
for  $L\in\Irr(\Fin(\fp_n))$ if $x$ is of rank $1$,
see~\cite{EAS}.

In~\cite{GS}, C.~Gruson and V.~Serganova  express
the character of a simple finite-dimensional $\osp(M|N)$-module
in terms of a basis consisting of ``Euler characters''. Using $\pari$
we will show that all coefficients in such formula have the same sign (equal to
$\pari(L)$) (we call this property ``positiveness'').
The similar formulae hold for the simple modules in $\Fin(\fg)$, where
$\fg$ is an exceptional Lie superalgebra or  $\fgl(1|n)^{(1)}$ (see~\cite{KW2},\cite{S3}). For $\fq_n$ the character formula of the above form was obtained in~\cite{PS1},\cite{PS2} (see also~\cite{Br}).
In~\cite{GH2} we will prove
a Gruson-Serganova type  character formula for $\fgl(m|n)$-case.

The reduced Grothendieck ring 
is the quotient of the Grothendieck ring
modulo the relation $[N]+[\Pi N]=0$. If $\CC$ is rigid,
then $*$ induces an involution of the reduced Grothendieck ring. 
By Hinich's Lemma, $\DS$-functor induces a homomorphism $\ds$
of the reduced Grothendieck rings; this homomorphism, introduced in~\cite{HW}, 
is compatible with the above involutions (in many cases the reduced Grothendieck ring is isomorphic to 
the ring of supercharacters so 
$\ds$ can be represented as the restriction of supercharacters
to a subalgebra of $\fh$; for the algebras from the list~(\ref{list}) 
the homomorphism $\ds$ was studied in~\cite{HR}).


\subsection{The map $\pari$ for finite-dimensional Kac-Moody superalgebras}\label{subsectpari}

Let $\fg$ be one of  the following superalgebras:

\begin{equation}\label{list}
\fgl(m|n), \osp(m|2n)\ \text{ for  }m,n\geq 0,\ D(2,1|a),\ 
F(4), \ G(3),\ \fsl(m|n)
\ \text{ for } m\not=n.
\end{equation}

Note that the  list~(\ref{list})  includes $\fgl(0|0)=\osp(1|0)=\osp(0|0)=0$, 
$\osp(2|0)=\mathbb{C}$ and the reductive Lie algebras $\fgl_m,
\mathfrak{o}_m,\mathfrak{sp}_m$. 
For each value of $x$, the algebra $\fg_x$ is again one of the algebras from the list~(\ref{list}).

Let $\fh$ be a Cartan subalgebra of $\fg$. 
We denote by $\Lambda_{m|n}$ the integral weight lattice in $\fh^*$;
this lattice is 
equipped by the standard parity function $p:\Lambda_{m|n}\to\mathbb{Z}_2$
(for $\fgl(m|n),\osp(M|N)$ and $G(3)$ the lattice $\Lambda_{m|n}$ is spanned
by $\vareps_i$s and $\delta_j$s 
with
$p(\vareps_i):=\ol{0}$ and $p(\delta_j):=\ol{1}$).
For our purposes  the study of the category $\Fin(\fg)$ of finite dimensional representations of $\fg$  reduces to study the category $\tilde{\cF}(\fg)$  with the modules whose weights lie in
 $\Lambda_{m|n}$. 
In its turn, the category $\tilde{\cF}(\fg)$ decomposes into a direct sum two equivalent categories  $$\tilde{\cF}(\fg) = \cF(\fg) \oplus \Pi (\cF(\fg)),$$
where the grading on the modules in $\cF(\fg)$ is induced by the parity function $p$.
Our goal is to  find a map
$$\pari: \Irr(\tilde{\cF}(\fg))\to \mathbb{Z}_2$$ satisfying
 (Dex1) and (Dex2).  By~\cite{Skw}, for each $L\in\Irr(\tilde{\cF}(\fg)))$
there exists $x$ such that  $\DS_x(L)$ is a non-zero typical module.
Therefore it is enough to define
$\pari$ on the typical simple modules. If   this is done in  such a way that
$\pari(\Pi(L))\not=\pari(L)$, then $\pari$ satisfying (Dex2) is unique and
satisfies
\begin{equation}\label{dexPi}
\pari(\Pi(L))\not=\pari(L) \ \ \text{ for each }\  L\in  \Irr(\tilde{\cF}(\fg)).
\end{equation}

\subsection{Reduction to $\DS_1$}\label{checkDex2}
In many cases it is enough to check (Dex2) for one particular value of $x$. 
We continue to consider the case when $\fg$ is as in~(\ref{list}) (the same reasoning work
for symmetrizable affine Lie superalgebras for $x$ as in Section 9 of~\cite{Gdepth}). We set 
$$\rank x:=\defect \fg-\defect \fg_x \ \ \ \text{ for }x\in X(\fg).$$
By~\cite{DS}, 
$\fg_x\cong \fg_{y}$ if $x,y\in X(\fg)$
are such that $\rank x=\rank y$; we set
$$X(\fg)_r:=\{x\in X(\fg)|\ \rank x=r\}.$$

 Using Lemma 2.4.1 in~\cite{Gdepth},
one can reduce (Dex2) to the case $x\in X(\fg)_1$
(see~\cite{GH}, Section 9 for details).
Fix any $x\in X(\fg)_1$ 
and denote $\DS_x$ by $\DS_1$. Using the results of~\cite{DS},
it is easy to see
 that for each $y\in X(\fg)_1$   there exists
 an automorphism $\phi\in \Aut(\fg)$ satisfying
$\phi(x)=y$; this automorphism induces
an isomorphism $\ol{\phi}: \fg_x\iso \fg_{y}$ and
$\DS_{x}(N^{\phi})=(\DS_y(N))^{\ol{\phi}}$. Thus $\DS_1$ 
is ``independent'' from the choice of $x$; in particular,
if the formula (Dex2) holds for $\DS_x$, then it holds for each $y\in X(\fg)_1$.
The argument of~\cite{GH}, Section 9 give  

{\em $\ \ \ \ \ $ if $\pari$ satisfies}
(Dex1), (\ref{dexPi}) {\em and  } (Dex2) {\em holds for some $x$ of rank $1$, 
then }
\begin{equation}\label{111}
\DS_{x'}(L)\cong \underbrace{\DS_1(\DS_1\ldots \DS_1}_{\rank x' \text{ times}}(L)\ldots)\ \ \ \text{ for any } L\in\Irr(\Fin(\fg))  \text{ and each } x'\in X(\fg).\end{equation}

In this way, the computation of $\DS_x(L)$ reduces
to the computation of the multiplicities  $[\DS_1(L'):L'']$ for
each quadruple $(\fg',\fg'', L', L'')$, where 
$$\fg':=\underbrace{\DS_1(\DS_1\ldots \DS_1}_{i \text{ times}}(\fg)\ldots ), \ \ 
\fg'':=\DS_1(\fg')$$
and $L'\in\Irr(\Fin(\fg'))$, $L''\in\Irr(\Fin(\fg''))$.

\subsubsection{}
The multiplicities $[\DS_1(L'):L'']$ were computed
for $\fgl(m|n)$ in~\cite{HW}.
For the exceptional cases we compute the multiplicities
in Section~\ref{sectatyp1}. For the remaining case $\osp(m|n)$
the multiplicities are computed 
in~\cite{GH}. In all these cases the following properties hold:
\begin{itemize}
\item
the module $\DS_1(L')$ is pure;
\item 
 $[\DS_1(L'):L'']\leq 2$;
\item
$[\DS_1(L'):L'']\leq 1$ for $\fg=\fgl(m|n)$;
\item
there exists $\pari$ satisfying (Dex1), (\ref{dexPi}) and the formula
(Dex2) for the case when $\rank x=1$.
\end{itemize}
By above, this implies (Dex2) (for any $x\in X(\fg)$) and shows that 
 $\DS_x(L)$ is pure, semisimple and
can be computed via the formula~(\ref{111}).

\subsubsection{}
In~\ref{caseF(4)} we compute $\DS_1(L)$ for $F(4)$. The results imply
that the image of the Grothendieck ring of $\cF(F(4))$ under the homomorphism $\ds$  coincides with $\sigma$-invariants in the Grothendieck ring of $\cF(\fsl_3)$ for $\sigma$ induced by a Dynkin diagram involution 
of $\fsl_3$.  For all other algebras from the list~(\ref{list}) a similar result is obtained 
 in~\cite{HR}.
 
\subsubsection{Remark}
By~\cite{EAS}, 
$\DS_1(L)$ is pure and multiplicity free for each $L\in \Irr(\Fin(\fp_n))$;
however,  $\DS_1(L)$ is not always semisimple, (\ref{111}) does not hold
and $\DS_1(\DS_1(L))$ is not always pure 
(see~\cite{Gdepth}, Example 3.4.3).

\subsubsection{Question}\label{conjecture}
Let $\fp\subset\fg$ be a parabolic subalgebra containing $\fb$
with the property that the defect of the Levi subalgebra of  $\fp$
is less by one than the defect of $\fg$. We denote by $L(\lambda)$
(resp., $L_{\fp}(\lambda)$) a simple  $\fg$ (resp., $\fp$) module
of the highest weight $\lambda$. For $L(\lambda)\in\tilde{\cF}(\fg)$
we define $\Gamma^i_{\fg,\fp}(L_{\fp}(\lambda))$ as in~\cite{GS} (see~\ref{GammaiGS}
below) and consider the multiplicities
$$m^{(i)}_{\lambda,\mu}:=[\Gamma^i_{\fg,\fp}(L_{\fp}(\lambda)):L_{\fg}(\mu)].$$
One has $m^{(0)}_{\lambda,\lambda}=1$.
In all our examples  $m^{(i)}_{\lambda,\mu}\in\{0,1\}$
and, except for the case $m^{(0)}_{\lambda,\lambda}$, one has
$$m^{(i)}_{\lambda,\mu}\not=0\ \Longrightarrow\ 
\pari(L(\lambda))-\pari(L(\mu))\equiv i+1 \mod 2.
\ $$

It is interesting to know whether these properties hold in other cases.

\subsection{Content of the paper}
In Section~\ref{sectionDS} we recall  the construction of $\DS$-functor.

In Section~\ref{sectatyp1} we consider the cases $\fg=D(2|1;a)$, $F(4)$, $G(3)$ and $\osp(3|2)$. We compute $\DS_x(L)$ for $L\in\Irr(\tilde{\cF}(\fg))$ and
check that $\pari$ 
satisfies (Dex1),(Dex2) and~(\ref{dexPi}). The modules $\DS_x(L)$ can be described as follows.
To each atypical
block $\cB$ in $\tilde{\cF}(\fg)$ we  assign a $\fg_x$-module $L'$ (this assignment is injective).
By~\cite{Germoni2}, for
 each atypical block of  $\osp(3|2), G(3)$  the extension graph 
$\Ext(\cB)$ if 
$D_{\infty}$; for  the cases $F(4)$, $D(2|1;a)$ the graphs of atypical blocks 
are $A^{\infty}_{\infty}$ and $D_{\infty}$
 (see~\cite{Germoni2}, \cite{Lilit}). 
If $\Ext(\cB)=D_{\infty}$, then $L'$ is simple and
$\DS_x(L)\cong L'$ if $L\in \Irr(\cB)$ is an ``end vertex''
of $D_{\infty}$ and $\DS_x(L)\cong \Pi^{i}(L')^{\oplus 2}$
if $L$ is the $i$th vertex counting from the ends. For $\Ext(\cB)= A^{\infty}_{\infty}$ one has
$\DS_x(L)\cong \Pi^i(L')$ for each $L\in\Irr(\cB)$.
 For 
$\fg=F(4)$, $D(2|1;a)$ 
 the module $L'$ corresponding to $D_{\infty}$-graph is a simple
$\fg_x$-module satisfying $(L')^*\cong L'$, whereas 
the module $L'$ corresponding to $A_{\infty}^{\infty}$-graph 
is of the form  $L'=V\oplus V^*$ , where $V$ is a simple
$\fg_x$-module with $V^*\not\cong V$.

In Section~\ref{Gamma} we construct for each block $\cB$ 
in $\tilde{\cF}(\fg)$ a graph $\hat{\Gamma}^{\chi}$ and its subgraph ${\Gamma}^{\chi}$ defined in terms
of the functors $\Gamma^i$ introduced in~\cite{Penkov} and~\cite{GS}
(we follow the definition in~\cite{GS}). The extension graph
 $\Ext(\cB)$ is a  a subgraph of  $\hat{\Gamma}^{\chi}$; the graph
${\Gamma}^{\chi}$ is useful for  Gruson-Serganova type character formulae.
In~\ref{bipartition} we introduce a notion  of ``parametric bipartition''
on the graphs $\hat{\Gamma}^{\chi}$, ${\Gamma}^{\chi}$; a parametric bipartition on $\hat{\Gamma}^{\chi}$ induces a bipartition on  $\Ext(\cB)$;
a parametric bipartition on ${\Gamma}^{\chi}$ gives 
 a ``positive'' Gruson-Serganova type character formulae.  In~\Cor{maincor}
we show that under a certain conditions (which hold in the $\osp$-case)
 $\Ext(\cB)$ is a  a subgraph of  ${\Gamma}^{\chi}$ and a 
parametric bipartition on ${\Gamma}^{\chi}$ induces a bipartition on  $\Ext(\cB)$.

The graphs $\hat{\Gamma}^{\chi}$, ${\Gamma}^{\chi}$ depend on the choice of triangular decomposition;
for the case $\osp(2|2)\cong \fsl(2|1)$
the graph
$\Ext(\cB)$ is a subgraph of ${\Gamma}^{\chi}$ for the ``mixed''
triangular decomposition and is not a  subgraph of ${\Gamma}^{\chi}$
for the distinushied one, see~\Rem{rem2221}.
By~\cite{MS}, for $\fgl(m|n)$-case the map $\pari$ gives a parametric biparition  on $\hat{\Gamma}^{\chi}$; we check that this also holds for 
 $\osp(2|2)$, $\osp(3|2)$,  $D(2|1;a)$, $F(4)$ and $G(3)$.

In Section~\ref{sectosp} we consider the principal block $\cB$ for 
$\fg=\osp(2n+t|2n)$ with $t=0,1,2$. In this case
the extension graph of $\cB$ is a subgraph of ${\Gamma}^{\chi}$. We describe $\pari$ which gives a parametric biparition on ${\Gamma}^{\chi}$;
this implies (Dex1) and 
 ``positiveness'' of the Gruson-Serganova  character formula.

In Section~\ref{arcs} we give a description of $\Gamma^{\chi}$ in the $\osp$-case
using the language of ``arch diagrams'' introduced in~\cite{GH}. 
The results of this section are not used in the rest of the paper.

In Appendix we explain why $\Ext(\cB)$ is a subgraph
of  $\hat{\Gamma}^{\chi}$ (this part essentially follows Sect. 6 of~\cite{MS}).

\subsection{Acknowledgments} The author is grateful to  I.~Entova-Aizenbud, V.~Hinich, T.~Heidersdorf, C.~Hoyt,  
I.~M.~Musson, S.~Reif, V.~Serganova, A.~Sherman and C.~Stroppel for numerous 
helpful discussions.

\subsection{Index of definitions and notation} \label{sec:app-index}
Throughout the paper the ground field is $\mathbb{C}$; 
$\mathbb{N}$ stands 
 for the set of non-negative integers. We will use the standard
Kac's notation~\cite{KLie}  for the root systems. 

\begin{center}
\begin{tabular}{lcl}
 $\pari$, (Dex1), (Dex2), pure & & \ref{introDex}\\

$\Lambda_{m|n},\ \tilde{\cF}(\fg), \cF(\fg)$   & & \ref{subsectpari}\\

$\Gamma^i_{\fp,\fq},\ ^iK^{\lambda,\nu}_{\fp,\fq}$ & & \ref{GammaiGS} \\


 increasing/descreasing paths & & \ref{decrincr}\\

decreasingly equivalent & & \ref{decreq}\\

(BB) & & \Lem{lemAA}\\

$\Lambda^{\chi}$, $\tail(\lambda)$, $\fp_{\lambda}$  & & \ref{usefulgraphs}\\

$\hat{\Gamma}^{\chi}$, $\kappa$, & & \ref{tildeGammachi}\\

 $\Gamma^{\chi}$ , (Tail) & & \ref{Gammachi}\\

 parametric bipartition & & \ref{bipartition}\\

$\mathcal{E}_{\lambda}$ & & (\ref{eqE}) in~\ref{twographs}\\

weight diagram & & \ref{wtdiag}\\

$Diag_{k;t}$, $\lambda(f)$ & &  \ref{diag}\\

$\tau$ & & \ref{tau}\\

$\pari$ for $\osp$-case, 

$||\lambda||$, $||\lambda||_{gr}$ & &\ref{paridef}\\

\end{tabular}
\end{center}

\section{$\DS$-functor}\label{sectionDS}
The $\DS$-functor was introduced in~\cite{DS}; see also~\cite{GHSS} for an expanded exposition. We recall definitions
and some results below. In this section $\fg$ is any superalgebra; we set  $X(\fg):=\{x\in \fg_1|\ [x,x]=0\}$.

\subsection{Construction}\label{DSconst}
For a $\fg$-module $M$ and $g\in\fg$ we set
$M^g:=\Ker_M g$. For $x\in X(\fg)$ we set
$$\DS_x(M):=M^x/xM.$$
Notice that $\fg^x$ and 
$\fg_x:=\DS_x(\fg)=\fg^{x}/[x,\fg]$
 are  Lie superalgebras. Since $M^x, xM$ are $\fg^{x}$-invariant and $[x,\fg] M^x\subset xM$,  $\DS_x(M)$ is a $\fg^{x}$-module and $\fg_x$-module. 
Thus 
$$\DS_x: M\mapsto \DS_x(M)$$ 
is a functor from the category of $\fg$-modules to
the category of $\DS_x(\fg)$-modules.

There are canonical isomorphisms $\DS_x(\Pi(N))=\Pi(\DS_x(N))$ and
$$\DS_x(M)\otimes\DS_x(N)=\DS_x(M\otimes N).$$
For a finite-dimensional module $L$ one has $\DS_x(L^*)\cong \DS_x(L)^*$.

\subsection{Hinich's Lemma}\label{Hinich}

The following result is called {\em Hinich's Lemma} (see~\cite{DS}); a similar result is Lemma 2.1 in~\cite{HW}.

A short  exact sequence of $\fg$-modules
$$0\to M_1\to N\to M_2\to 0$$
induces a long exact  sequence of $\fg_x$-modules
$$0\to Y\to \DS_x(M_1)\to\DS_x(N)\to\DS_x(M_2)\to \Pi(Y)\to 0$$
where $Y$ is some $\fg_x$-module.
We will use Hinich's Lemma in the following situation.

Let $N$ be a $\fg$-module with a three-step filtration
$$0=F_0(N)\subset F_1(N)\subset F_2(N)\subset F_3(N)=N$$
with the quotients $M_1,M_2,M_3$ (where $M_i:=F_i(N)/F_{i-1}(N)$).
%

\subsubsection{}
\begin{cor}{corHi}
If $\DS_x(N)=0$ and
$$\Hom_{\fg_x}(\DS_x(M_1),\Pi(\DS_x(M_3))=0,$$
then there exists an exact sequence
$$0\to \Pi(\DS_x(M_3))\to \DS_x(M_2)\to \Pi(\DS_x(M_1))\to 0.$$
\end{cor}
\begin{proof}
Since $\DS_x(N)=0$, the
Hinich's Lemma  gives an exact sequence
$$0\to Y_1\to \DS_x(M_1)\to 0\to \DS_x(N/M_1)\to\Pi(Y_1)\to 0$$
which implies $\DS_x(N/M_1)\cong \Pi(\DS_x(M_1))$.
Using the Hinich's Lemma for $N/M_1$ we obtain an exact sequence
$$0\to Y_2\to \DS_x(M_2)\to \Pi(\DS_x(M_1))\overset{\psi}{\longrightarrow}
\DS_x(M_3)\to \Pi(Y_2)\to 0.$$
By the assumption, $\psi=0$, so  $\DS_x(M_3)\cong \Pi(Y_2)$
which gives the required exact sequence.
\end{proof}

\section{The map $\pari$ for $\fg$ of defect $1$}\label{sectatyp1}
The simplest non-trivial extension graphs  are $A_{\infty},
A^{\infty}_{\infty}$ and $D_{\infty}$:
\begin{equation}\label{AAD}
\begin{array}{lrl}
A_{\infty}: & & L^0- L^1- L^2-L^3-
L^4-\ldots\\
A^{\infty}_{\infty}: & \ldots & L^{-2}-L^{-1}-L^0-L^1-
L^2-\ldots\\
\\
D_{\infty}: & L^1- &L^2-L^3-
L^4-\ldots\\
& & |\\
& & L^0\end{array}\end{equation}
(we depict $\longleftrightarrow$ by $-$). 
 J. Germoni conjectured that the extension graph of
each blocks of atypicality $1$ for a basic
classical Lie superalgebra is either $A^{\infty}_{\infty}$
or $D_{\infty}$; this conjecture was checked in~\cite{Germoni1},\cite{Germoni2} for all cases except $F(4)$, which was completed in~\cite{Lilit}.

\subsection{}
\begin{prop}{propdef1}
Take a non-zero $x\in\fg_1$ satisfying $x^2=0$.

Let $\cB$ be a block  and $\Irr(\cB)=:\{L^i\}_{i\in I}$. 
We assume that each $L^i\in\Irr(\cB)$ has a projective cover
with a three step radical filtration with the following subquotients

\begin{equation}\label{projdef1} 
L^i;\ 
\displaystyle\bigoplus_{j\in Adj(i)}\ L^j;\  
L^i,\ \ \text{ where }\ Adj(i):=\{j\in I|\ \Ext^1(L^i,L^j)\not=0\}.
\end{equation}
 
We set 
 $M_i:=\DS_x(L^i)$.
Assume that for some $s$ the module $M_s$ is pure
and
$$\Ext^1_{\tilde{\cF}(\fg_x)}(M_s,M_s)=0.$$ 
\begin{enumerate}
\item

For each $i\in I$ the module $M_i$ is pure.
Moreover, if $M_s\not=0$, then $M_i\not=0$ for each $i\in I$.
\item
 If $M_s\not=0$, then the extension graph $\Ext(\cB)$ is bipartite.
\item Using the notation of~(\ref{AAD}) we have
$$\begin{array}{lll}
\text{ if }\Ext(\cB)=A_{\infty} & \text{then } & 
M_{j}\cong\Pi^{j}(M_0)^{\oplus 2}\ \text{ for }j\geq 1;\\
\text{ if }\Ext(\cB)=D_{\infty} & \text{then } & 
M_{j}\cong\Pi^{j-1}(M_0)^{\oplus 2}\ \text{ for } j\geq 2,\  M_1\cong M_0;\\
\text{ if }\Ext(\cB)=A_{\infty}^{\infty} & \text{then } & 
M_j\cong \Pi^{j}(M_0).\end{array}$$

\end{enumerate}
\end{prop}
\begin{proof}
By~\cite{DS}, the functor $\DS_x$ (for $x\not=0$)  kills the projective modules in $\cF(\fg)$. Using~\Cor{corHi} and~(\ref{projdef1})  we conclude that for any $i$
the purity of  $M_i$ implies the existence of
an exact sequence
\begin{equation}\label{adj}
0\to \Pi(M_i)\to \displaystyle\bigoplus_{j\in Adj(i)} M_j\to \Pi(M_i)\to 0
\end{equation}
In particular, the purity of  $M_i$ implies implies the purity of $M_j$ for
each  $j\in Adj(i)$ and $M_i=0$ implies $M_j=0$ for  $j\in Adj(i)$.
Since $\Ext(\cB)$ is connected,  this proves (i).

Let $L'$ be a simple module $L'$ such that $[M_s:L']\not=0$.
For each $i\in I$
set $p_i:=[M_i:L']$, $q_i:=[M_i:\Pi(L')]$. By above, $p_s\not=0$.
Using~(\ref{adj})  and (i) we obtain 
\begin{equation}\label{madj}
p_i,q_i\geq 0,\ \ \ 2q_i=\sum_{j\in Adj(i)} p_j,\ \ 2p_i=\sum_{j\in Adj(i)} q_j,\ \ p_iq_i=0
\end{equation}
for each $i\in I$. 
In particular, if $p_j=q_j=0$ for some $j$, then $p_i=q_i=0$ for each $i$, a contradiction.
Hence for each $i$ either $p_i\not=0$ or $q_i\not=0$.
It is easy to see from~(\ref{madj}) that $p_i=0$ if $L^i,L^s$ are connected by a path of odd length. Hence $\Ext(\cB)$ does not have cycles of odd length;
this gives (ii). For $A_{\infty}$ and $D_{\infty}$, (iii) follows from~(\ref{madj}) by induction.
For the case $A^{\infty}_{\infty}$ observe that  $m_i:=p_i+q_i$ satisfies
$2m_i=m_{i-1}+m_{i+1}$ for $i\in I=\mathbb{Z}$. Since $m_i\geq 0$,  we get
$m_i=m_s$ for each $i$.
This completes the proof.
\end{proof}

\subsection{$\DS$-functor for small rank $\fg$}\label{smallrank}
Let $\fg$ be one of the Lie superalgebras
$$\osp(2|2),\osp(3|2),  D(2|1;a), G(3), F(4).$$
Let $\fh$ be a Cartan subalgebra of $\fg_0$.
We denote by  $W$  the Weyl group of $\fg_0$ and  by $(-|-)$ the symmetric non-degenerate form on $\fh^*$ which is induced by a non-degenerate invariant
form on $\fg$.

\subsubsection{}
Let $\Sigma$ be a base of $\fg$ which contains an isotropic root $\beta$.
Fix a non-zero $x\in\fg_{\beta}$. 

Set
${\Delta}_x:=(\beta^{\perp}\cap{\Delta})\setminus\{\beta,-\beta\}$.
By~\cite{DS},  ${\fg}_x$ can be identified
with a subalgebra of ${\fg}$ generated by the root spaces
$\fg_{\alpha}$ with $\alpha\in{\Delta}_x$ and a Cartan subalgebra
${\fh}_x\subset\fh$. If ${\Delta}_x$ is not empty, then
${\Delta}_x$ is the root system of the Lie superalgebra ${\fg}_x$ and one can choose
${\Sigma}_x$ in $\Delta_x$ such that 
$\Delta^+({\Sigma}_x)=\Delta^+\cap {\Delta}_x$.
If ${\fg}=\osp(m|2)$, then ${\fg}_x=\mathfrak{o}_{m-2}$; for 
 $\fg=D(2|1;a),G_3,F_4$  one has   ${\fg}_x=\mathbb{C},\fsl_2,\fsl_3$ respectively.
 
\subsubsection{}
 \begin{lem}{tame0}
 Let $L:=L(\lambda)$ be a finite-dimensional module and $(\lambda|\beta)=0$.
 Set $L':=L_{\fg_x}(\lambda|_{\fh_x})$. One has
$$\DS_x(L)\cong \left\{\begin{array}{lll}
L' &\text{ for } & \osp(2|2),\osp(3|2), G(3)\\
L' &\text{ for } & D(2|1;a),F(4) \text{ if }\  L'\cong (L')^*\\
L'\oplus (L')^*&\text{ for } & D(2|1;a),F(4) \text{ if }\  L'\not\cong (L')^*.
\end{array}\right.$$
 \end{lem} 
    \begin{proof}
It is easy to see that $[\DS_x(L):L']=1$.     
Set $\lambda':=\lambda|_{\fh_x}$. 
From~\cite{DS}, Sect. 7,
    $\DS_x(L)$ is a typical module and each simple subquotient of $\DS_x(L)$
is of the form
$L_{\fg_x}(\nu)$ with $\nu\in\{\lambda', \sigma(\lambda')\}$, where $\sigma=\Id$
for $\fg=\osp(2|2),\osp(3|2)$ and $G(3)$, $\sigma=-\Id$ for $D(2|1;a)$ and
$\sigma$ is the Dynkin diagram automorphism of $\fg_x=\fsl_3$ in $F(4)$-case. 
 This gives the first formula.
For $D(2|1;a), F(4)$ one has $L_{\fg_x}(\nu)^*\cong L_{\fg_x}(\sigma(\nu))$;
this gives the second formula.  For $\fg=D(2|1;a), F(4)$
   the Weyl group contains $-\Id$, so $L\cong L^*$ and thus
$\DS_x(L)\cong \DS_x(L^*)$ by~\ref{DSconst}.
   This implies the third formula.
   \end{proof}

\subsubsection{}\label{firsttype}
We fix a triangular decomposition of $\fg_0$ and denote by $\Delta^+_0$ the corresponding set of positive roots. We consider 
 all bases $\Sigma$ for $\Delta$ which satisfy
$\Delta^+_0\subset \Delta^+(\Sigma)$.
We say that an isotropic root $\beta$ is of the {\em first type}
if $\beta$ lies in  a base $\Sigma$ with $\Delta^+_0\subset \Delta^+(\Sigma)$.

Take any base $\Sigma$ as above and denote by $\rho$ the corresponding Weyl vector.
It is easy to see that a simple atypical module $L=L(\nu)$  satisfies the assumptions of~\Lem{tame0} for some $\Sigma'$ and $\beta\in\Sigma'$ 
if and only if $\nu+\rho$ is orthogonal
to an isotropic root of the first type.

\subsection{Blocks of atypicality $1$}\label{DSat1}
The blocks of atypicality $1$ for basic classical Lie superalgebras
were studied by J. Germoni in~\cite{Germoni1},\cite{Germoni2}
and by L.~Martirosyan in~\cite{Lilit}.  These blocks
satisfy the assumption
 of~\Prop{propdef1} and have the following extension graphs:

$A^{\infty}_{\infty}$ for $\fg=\fgl(m|n), \osp(2m|2n), F(4), D(2|1;a) \text{ for } a\in\mathbb{Q};$

$D_{\infty}$ for $\fg=\osp(2m+1|2n),\osp(2m|2n),  F(4), G(3), D(2|1;a)$.

Let $\fg$ be as in~\ref{smallrank} and  $\cB$ be  an atypical block.

 We call a block containing the trivial module  $L(0)$
a {\em principal block}. Clearly, $\DS_x(L(0))$ is the trivial $\fg_x$-module,
so~\ref{propdef1} gives  $\DS_x(L)$ for each  $L\in\Irr(\cB_0)$.
For $\osp(2|2),\osp(3|2)$ the  principal block is the only atypical block.

Combining~\ref{propdef1} and~\ref{tame0}, \ref{firsttype}
 we see that in order to compute
$\DS_x(L)$ for each  $L$ in $\Irr(\cB)$, it is enough to find
$L(\nu)\in\Irr(\cB)$ such that $\nu+\rho$ is orthogonal to  an isotropic root of the first type.
Below we will list such $\nu$ for each non-principal atypical block in
 the remaining cases $D(2|1;a)$, $F(4)$ and $G(3)$.

%
%

\subsubsection{Case $D(2|1;a)$}\label{D21a}
For $\fg:=D(2|1;a)$ one has $\fg_x=\mathbb{C}$.
The atypical blocks were described in~\cite{Germoni2}, Thm. 3.1.1.

The extension graph of the principal block $\cB_0$ is $D_{\infty}$, so for $L^i\in\Irr(\cB_0)$ we have
$\DS_x(L^i)=\mathbb{C}$
for $i=0,1$ and $\DS_x(L^i)=\Pi^{i-1}(\mathbb{C})^{\oplus 2}$
for $i>1$ (where $\mathbb{C}$ stands for the trivial even $\fg_x$-module).

If $a$ is irrational, the principal block is
 the only atypical block in $\cF(\fg)$. Consider the case when $a$ is rational.
 Recall that  $\fh^*$ has an orthogonal basis
$\{\vareps_1,\vareps_2,\vareps_3\}$ with 
$$||\vareps_1||^2=\frac{a}{2},\ \ \ 
||\vareps_2||^2=\frac{1}{2},\ \ \ ||\vareps_3||^2=-\frac{1+a}{2};$$
let $\vareps_1^*,\vareps_2^*,\vareps_3^*$ be the dual basis in $\fh$.
The lattice $\Lambda_{2|1}$ is spanned by $\vareps_i$s; the parity map  is given by
$p(\vareps_1)=p(\vareps_2)=\ol{0}$, $p(\vareps_3)=\ol{1}$.
One has
$$D(2|1;1)=\osp(4|2),\ \ D(2|1;a)\cong D(2|1; -1-a)\cong D(2|1; a^{-1})$$
so we can assume that  $0<a<1$ and write  $a=\frac{p}{q}$, where $p,q$ are relatively prime positive integers. 

The atypical blocks  are $\cB_k$ for  $k\in\mathbb{N}$ 
(the principal block is $\cB_0$). Consider the block $\cB_k$ with 
 $k>0$. The extension graph of $\cB_k$ is
$A^{\infty}_{\infty}$. By~\cite{Germoni2}, Thm. 3.1.1, the block 
$\cB_k$ contains a simple module $L$ with the highest weight $\lambda_{k;0}$
satisfying $(\lambda_{k;0}+\rho|\beta)=0$ for 
$$\beta:=-\vareps_1+\vareps_2+\vareps_3.$$
Taking $x\in\fg_{\beta}$ we can
 identify $\fg_x$ with $\mathbb{C}h$ for
 $h:=q\vareps_1^*+p\vareps_2^*$.
Combining~\ref{tame0} and~\ref{firsttype} we get
$$\DS_x(L)=L_{\fg_x}(k)\oplus L_{\fg_x}(-k),$$
where $L_{\fg_x}(u)$ stands for the even one-dimensional $\fg_x$-module
with $h$ acting by $u(p^2+q^2)$.
By~\Prop{propdef1}, $\DS_x(L^i)\cong \Pi^i(\DS_x(L))$ 
for each $L^i\in\Irr(\cB_k)$ (for $k>0$).

\subsubsection{Case $G(3)$}
For $\fg:=G(3)$ the atypical blocks were described in~\cite{Germoni2}, Thm. 4.1.1.
The atypical blocks in $\cF(\fg)$ are $\cB_k$ for  $k\in\mathbb{N}$;
the extension graphs are  $D_{\infty}$. The block 
$\cB_k$ contains a simple module with the highest weight $\lambda_{k;0}$
satisfying $(\lambda_{k;0}+\rho|\beta)=0$ for 
$$\beta:=-\vareps_1+\delta.$$
Taking $\Sigma:=\{\delta-\vareps_1,\vareps_2-\delta,\delta\}$ and 
$x\in\fg_{\beta}$ we can
 identify $\fg_x$ with
 $\fsl_2$-triple corresponding to the root $\alpha=\vareps_1+2\vareps_2$.
 One has $\lambda_{k;0}=k\alpha$.
Combining~\ref{tame0} and~\ref{propdef1} we get
$$\DS_x(L^0)\cong \DS_x(L^1)\cong L_{\fsl_2}(2k),\ \ \ 
\DS_x(L^i)=\Pi^{i-1}(L_{\fsl_2}(2k))^{\oplus 2}\ 
\text{ for }i>1.$$

 \subsubsection{Case $F(4)$}\label{caseF(4)}
For $\fg:=F(4)$ we have $\fg_x\cong \fsl_3$. 
The integral weight lattice is spanned by $\vareps_1$, $\vareps_2$,
$\frac{1}{2}(\vareps_1+\vareps_2+\vareps_3)$ and $\frac{1}{2}\delta$;
the parity is given by 
$p(\frac{\vareps_i}{2})=\ol{0}$ and $p(\frac{\delta}{2})=\ol{1}$.

The atypical blocks are described
in~\cite{Lilit}, Thm. 2.1. These blocks are
parametrized by the pairs $(m_1,m_2)$, where
$m_1,m_2\in \mathbb{N}$,  $m_1-m_2\in 3\mathbb{N}$.
We denote the corresponding block by $\cB_{(m_1;m_2)}$.

The extension
graph of $\cB_{(i;i)}$ is $D_{\infty}$; the block
$\cB_{(0;0)}$ is principal.
For $i>0$ the block 
$\cB_{(i;i)}$  contains a simple module $L(\lambda)$ with
 $$\lambda+\rho=(i+1)(\vareps_1+\vareps_2)-\beta_1, \ \text{ where }
 \beta_1:=
\frac{1}{2}(-\vareps_1+\vareps_2-\vareps_3+\delta).$$
One has
 $(\lambda+\rho|\beta_1)=0$. Take $x\in \fg_{\beta_1}$ and consider the base
 $$\Sigma_1:=\{\beta_1; \frac{1}{2}(\vareps_1+\vareps_2-\vareps_3-\delta); 
\vareps_3; \vareps_1-\vareps_2\}.
$$
Then
  $\fg_x$ can be identified with $\fsl_3$ corresponding to the set of simple roots
$\ \{\vareps_2+\vareps_3;\vareps_1-\vareps_3\}$ and~\Lem{tame0} gives
 $$\DS_x(L(\lambda))=L_{\fsl_3}(i\omega_1+i\omega_2),$$
 where $\omega_1,\omega_2$  are the fundamental weights of $\fsl_3$.
By~\ref{propdef1} we get for $L^j\in\Irr(\cB_{(i;i)})$:
 $$
 \DS_x(L^0)\cong \DS_x(L^1)\cong L_{\fsl_3}(i\omega_1+i\omega_2),\ \ 
\DS_x(L^j)\cong \Pi^{j-1}(L_{\fsl_3}(i\omega_1+i\omega_2))^{\oplus 2}\
 \text{ for } j>1.$$ 
 
Consider a   block  $\cB_{(i_1;i_2)}$  for $i_1\not=i_2$.
The extension
graph of this block is $A_{\infty}^{\infty}$  and this
 block  contains a simple module $L:=L(\lambda')$ with
$$\lambda'+\rho=i_1\vareps_1+i_2\vareps_2+(i_1-i_2)\vareps_3.$$
In particular, $(\lambda'+\rho|\beta_2)=0$ for
$\beta_2:=\frac{1}{2}(-\vareps_1+\vareps_2+\vareps_3+\delta)$.
Taking  $x\in\fg_{\beta_1}$ and
 $$\Sigma_2:=\{\beta_2; \vareps_2-\vareps_3;  -\beta_1; \frac{1}{2}(\vareps_1-\vareps_2-\vareps_3+\delta)\}$$
 we identify
  $\fg_x$ with $\fsl_3$ corresponding to the set of simple roots
$\{\vareps_2-\vareps_3;\vareps_1+\vareps_3\}$. Combining~\Lem{tame0} 
and~\ref{propdef1} we get 
 $$\DS_x(L)=L_{\fsl_3}(i_1\omega_1+i_2\omega_2)\oplus 
 L_{\fsl_3}(i_2\omega_1+i_1\omega_2),\ \ \DS_x(L^i)\cong \Pi^i(\DS_x(L))$$
for each $L^i$ in the block $\cB_{(i_1;i_2)}$.

\begin{cor}{corF4}
The image of the Grothendieck ring of $\cF(F(4))$ under the homomorphism $\ds$  coincides with $\sigma$-invariants in the Grothendieck ring of $\cF(\fsl_3)$.
\end{cor}
\begin{proof}
The condition $m_1-m_2$ divisible by $3$ is
equivalent to $m_1\omega_1+m_2\omega_2$ lies in the root lattice of $\fsl_3$.
\end{proof}


\subsection{Conclusion}\label{conclatyp1}
Let $\ft$ be one of the superalgebras in~\ref{smallrank} or one of Lie algebras
$0,\mathbb{C}, \fsl_2$ or
$\fsl_3$. We introduce
the map $\pari$ for $\ft$ by

\begin{enumerate}

  \item[$-$]
 for a typical $L\in\Irr(\tilde{\cF}(\ft))$ 
  we take $\pari(L):=0$ for   $L\in\Irr(\cF(\ft))$;
 
\item[$-$] 
  for an atypical $L\in\Irr(\tilde{\cF}(\ft))$ we set $\pari(L):=0$ if
   $\DS_x(L)$ is an even  vector space.

\item[$-$] $\pari(\Pi(L)):\equiv \pari(L)+1\ \mod 2.$
\end{enumerate}

One readily sees that $\pari$ satisfies (Dex1) and (Dex2).

\section{Functors $\Gamma^i_{\fg,\fq}$}\label{Gamma}
In this section $\fg$ is one of the superalgebras $F(4), G(3)$,
$\fgl(m|n),\osp(m|2n)$  for $m,n\geq 0$.
We fix any triangular decomposition  $\Delta=\Delta^+\coprod (-\Delta^+)$ and
denote by $\fb$ the corresponding Borel subalgebra. We consider the standard partial order $\nu_1\leq \nu_2$ for $\nu_2-\nu_1\in \mathbb{N}\Delta^+$.

\subsection{}\label{GammaiGS}
Let
$\fq\subset\fp$ be  parabolic subalgebras containing $\fb$
and $\fl$ be the Levy factor of $\fp$. For a finite-dimensional $\fq$-module $V$ denote by $\Gamma_{\fp,\fq}(V)$
 the maximal finite-dimensional
quotient of  the induced module $\cU(\fp)\otimes_{\cU(\fq)}V$.
We denote by $\tilde{\cF}(\fp)$
the category of finite-dimensional $\fp$-modules with the restriction
lying in $\tilde{\cF}(\fl)$ and by $\Ext^1_{\fp}$ the functor $\Ext^1$ in this category.  For $\lambda\in\Lambda_{m|n}$ we denote by $L_{\fp}(\lambda)$
a simple $\fp$-module of the highest weight $\lambda$ with
the grading induced by the parity function on $\Lambda_{m|n}$.

In~\cite{GS}, Sect. 3  the authors introduce for $i=0,1,\ldots$ an additive functor
$$\Gamma^i_{\fg,\fp}: \tilde{\cF}(\fp)\to \tilde{\cF}(\fg)$$
(in~\cite{GS} these functors are denoted by $\Gamma_i(G/P;-)$) in the following way.
For each $V\in \tilde{\cF}(\fq)$ we take the vector bundle
$G\times_P V$ over the generalized Grassmanian $G/P$ and consider the cohomology groups $H^i(G/P, G\times_P V)$ as $\fg$-module. We set
$$\Gamma^i_{\fg,\fp}(V):=(H^i(G/P, G\times_P V^*))^*.$$

Below we recall several  properties of the functors $\Gamma^i_{\fg,\fp}$;
for the proofs and other  properties see~\cite{GS}, Sections 3, 4.

\subsubsection{}\label{Gamma0} One has
$\Gamma^0_{\fg,\fp}(V)=\Gamma_{\fg,\fp}(V)$. For each $i$
the module $\Gamma^i_{\fg,\fp}(L_{\fp}(\lambda))$ has the same central character
as $L(\lambda)$.

\subsubsection{}\label{exact}
Each short exact sequence of $\fq$-modules
$$0\to U\to V\to U'\to 0$$
induces a long exact sequence
$$\ldots\to \Gamma^1_{\fg,\fq}(V)\to \Gamma^1_{\fg,\fq}(U')\to \Gamma^0_{\fg,\fq}(U)\to \Gamma^0_{\fg,\fq}(V)\to
\Gamma^0_{\fg,\fq}(U')\to 0.$$

\subsubsection{}\label{gammawi}
If $[\Gamma^i_{\fg,\fp}(L_{\fp}(\lambda)):L(\nu)]\not=0$, then
there exist $I\subset\Delta^+_1$ and $w\in W$ of length $i$ such that $\nu+\rho=w(\lambda+\rho)-\sum_{\alpha\in I}\alpha$.

\subsection{Poincar\'e polynomails}
Let $\fb\subset \fq\subset\fp$ be as in~\ref{GammaiGS}.
We set 
$$\fq':=\fq\cap\fl,\ \  \fh':=\fh\cap \fl,\ \ \fh''=\{h\in\fh|\ [h,\fl]=0\}$$
and notice that 
$\fh=\fh'\oplus\fh''$. Let 
$\lambda,\mu\in\Lambda_{m|n}$ be such that $L_{\fq}(\lambda)\in
{\cF}(\fq)$ and $L_{\fp}(\mu)\in \cF(\fp)$.
For $i=0,1,\ldots$ we define 

$$\ ^iK^{\lambda,\mu}_{\fp,\fq}:=\left\{ \begin{array}{ll}
\ \ 0\ & \text{ if }\lambda|_{\fh''}\not=\nu|_{\fh''};\\
\ [\Gamma^i_{\fl,\fq'}(L_{\fq'}(\lambda|_{\fh'})):L_{\fl}(\mu|_{\fh'})]\ & \text{ if }
\lambda|_{\fh''}=\nu|_{\fh''}\end{array}\right.$$
and introduce a Poincar\'e polynomial in the variable $z$ by
$$K^{\lambda,\mu}_{\fp,\fq}(z):=\sum_{i=0}^{\infty}\ ^iK^{\lambda,\mu}_{\fp,\fq} z^i.$$
(It is easy to see that this Poincar\'e polynomial is equal to
the Poincar\'e polynomial defined in~\cite{GS}, Section 4.)
When the term $K^{\lambda,\mu}_{\fp,\fq}$ appears in a formula it is always assumed that
$L_{\fq}(\lambda)\in{\cF}(\fq)$ and $L_{\fp}(\mu)\in \cF(\fp)$.

\subsubsection{}\label{Gamma0K}
By~\ref{Gamma0}  we have
$$[\Gamma_{\fp,\fq}(L_{\fq}(\lambda)):L_{\fp}(\mu)]=\ ^0K^{\lambda,\mu}_{\fp,\fq}=
K^{\lambda,\mu}_{\fp,\fq}(0).
$$
In particular, $K^{\lambda,\lambda}_{\fp,\fq}(0)=1$ and 
$K^{\lambda,\mu}_{\fp,\fq}(0)\not=0$ implies
$\mu\leq \lambda$. By~\cite{GS}, Thm. 1 one has
\begin{equation}\label{chainK}
K^{\lambda,\mu}_{\fg,\fq}(-1)=\sum_{\nu}K^{\lambda,\nu}_{\fp,\fq}(-1)
K^{\nu,\mu}_{\fg,\fp}(-1)\end{equation}
where the summation is taken on $\nu\in\fh^*$ with
$\dim L_{\fp}(\nu)<\infty$.

\subsection{Euler characteristic formula}\label{euler}
Let $\rho$ be the Weyl vector and $R$ be the Weyl denominator, i.e.
$$2\rho=\sum_{\alpha\in\Delta^+} (-1)^{p(\alpha)}\alpha,\ \ \ 
\ R=\prod_{\alpha\in\Delta^+} (1-(-1)^{p(\alpha)}e^{-\alpha})^{(-1)^{p(\alpha)}}.$$
We denote by $\sgn: W\to\mathbb{Z}_2$ the standard sign homomorphism and set
$$\mathcal{E}_{\lambda,\fp}:=R^{-1}e^{-\rho}\displaystyle\sum_{w\in W} \sgn(w) w\bigl(\frac{e^{\rho}\ch L_{\fp}(\lambda)}
{\displaystyle\prod_{\alpha\in\Delta_1^+(\fl)}(1+e^{-\alpha}) }  \bigr).$$

By~\cite{GS}, Prop.1, 
if    $L_{\fp}(\lambda)\in \cF(\fp)$, then
$$\sum_{\mu} K^{\lambda,\mu}_{\fg,\fp}(-1) \ch L(\mu)=\mathcal{E}_{\lambda,\fp}.$$

\subsubsection{}\label{remE}
Notice that $\mathcal{E}_{\lambda,\fp}$ can be zero.
For instance,  take  $\fg=\osp(m|2n)$ with  $m\geq 4$ 
and  $\fb$ coresponding to the  ``mixed base''. Then  
$\mathcal{E}_{0,\fb}=R^{-1}e^{-\rho}\sum \sgn(w) e^{w\rho}=0$. Since $\ch L(\mu)$ are linearly independent, we have
$K^{0,\mu}_{\fg,\fb}(-1)=0$ for all $\mu$.

%
%
%
%
%
%
%
%
%
%
%
%
%

\subsection{Marked graphs}\label{markedgraphs}
Consider a directed graph $(V,E)$ where $V$ is  at most countable
and the number of edges between any two vertices is finite.

We say that $\iota: V\to \mathbb{N}$  (resp., $\iota: V\to \mathbb{Z}$)
defines a $\mathbb{N}$-{\em grading} (resp., $\mathbb{Z}$-grading) 
on this graph if
for each edge $\nu\overset{e}{\longrightarrow} \lambda$ 
one has $\iota(\nu)<\iota(\lambda)$. Notice that for a $\mathbb{Z}$-graded graph
 the number of paths between any two vertices is finite.

Assume that
the set of edges $E$ is equipped by two functions $b$ and $\kappa$, where
$b: E\to\mathbb{Z}$ and $\kappa$ is a function from $E$ to a commutative ring.


\subsubsection{}\label{decrincr}
For a path
$P:=\nu_1\overset{e_1}{\longrightarrow }\nu_2\overset{e_2}{\longrightarrow }\nu_3 \ldots
\overset{e_s}{\longrightarrow }\nu_{s+1}$
we define
$$length(P):=s,\ \ \kappa(P):=\prod_{i=1}^s \kappa(e_i).$$

We call the path $P$
 {\em $b$-decreasing} (resp., {\em $b$-increasing}) if $b(e_1)>b(e_2)>\ldots >b(e_s)$
(resp., $b(e_1)<\ldots <b(e_s)$). 
We consider $P=\nu$  as a $b$-decreasing/increasing path  of zero length with  $\kappa(P)=1$. We denote the set of decreasing (resp., increasing) paths
from $\nu$ to $\lambda$ by $\cP^{>}_b(\nu,\lambda)$ (resp., $\cP^{<}_b(\nu,\lambda)$).

\subsubsection{}\label{decreq}
\begin{defn}{}
We call two functions $b,b': E\to\mathbb{Z}$ 
{\em decreasingly-equivalent} if
for each path
$\nu_1\overset{e_1}{\longrightarrow }\nu_2\overset{e_2}{\longrightarrow }\nu_3$ one has
$$b(e_1)>b(e_2)\ \Longleftrightarrow\ b'(e_1)>b'(e_2).$$
\end{defn}

Notice that $b,b'$  are descrearingly equivalent if and only if  
$\cP^{>}_b(\nu,\lambda)=\cP^{>}_{b'}(\nu,\lambda)$.

\subsubsection{}\label{Finass}
Let $(V,E)$ be a $\mathbb{Z}$-graded graph.
We introduce the square matrices $A^{<}(\kappa)=(a^{<}_{\lambda,\nu})_{\lambda,\nu\in V}\ $
and  $A^{>}(\kappa)=(a^{>}_{\lambda,\nu})_{\lambda,\nu\in V}\ $ by
$$a^{>}_{\lambda,\nu}:=\sum_{P\in \cP^{>}_b(\nu,\lambda)} \kappa(P),\ \ \ 
a^{<}_{\lambda,\nu}:=\sum_{P\in \cP^{<}_b(\nu,\lambda)}(-1)^{length(P)}\kappa(P).$$
Since  the graph is a $\mathbb{Z}$-graded, these matrices
are lower-triangular with  $a^{>}_{\lambda,\lambda}=a^{<}_{\lambda,\lambda}=1$.

\subsubsection{}
\begin{lem}{lemAA}
Let $(V,E)$ be a $\mathbb{Z}$-graded  graph with a finite number of edges
 between any two vertices. Assume that $b: E\to \mathbb{Z}$ satisfies the property

(BB) $ \ \ \ $
{\em for each path
$\ \nu_1\overset{e_1}{\longrightarrow }\nu_2\overset{e_2}{\longrightarrow }\nu_3\ $ one has $\ \ \ b(e_1)\not=b(e_2)$.} 

Then $A^{>}(\kappa)\cdot A^{<}(\kappa)=A^{<}(\kappa)\cdot A^{>}(\kappa)=\Id$.
\end{lem}
\begin{proof}
The proof is similar to~\cite{GS}, Thm. 4. The entries of $A^{<}(\kappa)\cdot A^{>}(\kappa)$
are of the form
$$\sum_{\mu}\sum_{P\in \cP^{<}_b(\nu,\mu)} \sum_{Q\in  \cP^{>}_b(\mu,\lambda)}
(-1)^{length(Q)}\kappa(PQ),$$
where $PQ$ stands for the concatenation of $P$ and $Q$.
The property (BB) implies 
that each  path of non-zero length which
can be presented as the concatenation $PQ$, where $P$ is 
$b$-increasing  and $Q$ is $b$-decreasing, 
has exactly two presentations of this form:
$PQ=P'Q'$ with $length\ Q'=length\ Q\pm 1$ 
(for instance, for a path of length $5$ with
$$b(e_1)=1,\ \  b(e_2)=2,\ \  b(e_3)=4,\ \  b(e_4)=2,\ \  b(e_5)=1$$
the increasing part can be  either $e_1,e_2$
or $e_1,e_2,e_3$). This implies the statement.
\end{proof}

\subsection{Useful graphs}\label{usefulgraphs}
We fix a sequence of parabolic subalgebras in $\fg$:
\begin{equation}\label{parabolicchain}
\fb=\fp^{(0)}\subset\fp^{(1)}\subset\ldots\subset \fp^{(k)}=\fg
\end{equation}
and denote by $\fl^{(p)}$ the Levy subalgebra of $\fp^{(p)}$.
We also fix a central character $\chi:\cZ(\fg)\to\mathbb{C}$ and denote by
$\Lambda^{\chi}$ the set of dominant weights corresponding to $\chi$:
$$\Lambda^{\chi}:=\{\lambda\in\Lambda_{m|n}|\ 
\dim L(\lambda)<\infty\ \text{ and } \Ann_{\cZ(\fg)} L(\lambda)=\Ker\chi\}.$$
For each $\lambda\in\Lambda^{\chi}$ we denote by 
$\tail(\lambda)$ the maximal $s$ such that
$\lambda|_{\fl^{(p)}\cap \fh}=0$. We set 
$$\fp_{\lambda}:=\fp^{\tail(\lambda)},\ \ \fl_{\lambda}:=\fl^{\tail(\lambda)}.$$

\subsubsection{Graph $\hat{\Gamma}^{\chi}$}\label{tildeGammachi}
Let $\hat{\Gamma}^{\chi}(z)$ be  a graph with the set of vertices
$V:=\Lambda^{\chi}$ and the following edges: if $K^{\lambda,\nu}_{\fp^{(s)},\fp^{(s-1)}}\not=\delta_{\nu,\lambda}$ (where $\delta_{\nu,\lambda}$ is the Kronecker symbol)
we join  $\nu,\lambda$ by 
the edge of  the form
$$\nu\overset{e}{\longrightarrow}\lambda\ \text{ with }\ \ b(e)=s,\ \ \ \kappa(e):=K^{\lambda,\nu}_{\fp^{(s)},\fp^{(s-1)}}(z)-\delta_{\nu,\lambda}\in \mathbb{Z}[z].$$
Note that each two vertices  in $\hat{\Gamma}^{\chi}$
are connected by at most $k$ edges.

For each $z_0\in\mathbb{C}$ we denote by $\hat{\Gamma}^{\chi}(z_0)$
the subgraph where the edges with $\kappa(e)(z_0)=0$ are deleted
and the function $\kappa_{z_0}: E\to\mathbb{C}$ is given by 
$$\kappa_{z_0}(e):=\kappa(e)(z_0).$$

Note that $\hat{\Gamma}^{\chi}(0)$ does not have loops, see~\ref{Gamma0K}.

%
%

\subsubsection{Graph ${\Gamma}^{\chi}$}\label{Gammachi}
Let $\Gamma^{\chi}$ (resp., $\Gamma^{\chi}(z_0)$) be the graph obtained
from $\hat{\Gamma}^{\chi}$ (resp., from $\hat{\Gamma}^{\chi}(z_0)$) 
by removing the edges of the form
$\nu\overset{e}{\longrightarrow} \lambda$
with $b(e)\leq \tail(\lambda)$. Thus we have

$\ \ \ \ \ \ \ \ \ \ 
\tail(\lambda)<b(e)\ $ {\em for each edge $\nu \overset{e}{\longrightarrow}\lambda$ in $\Gamma^{\chi}$}.

We will always assume that ${\Gamma}^{\chi}$ satisfies the following condition:

(Tail)$\ \ $ {\em $\tail(\nu)\leq b(e)\ $ for each edge $\nu \overset{e}{\longrightarrow}\lambda$ in $\Gamma^{\chi}$}

which is is tautological for $k=1$.

 \subsubsection{Notation}
 We denote by $\cP^{>}_b(\nu,\lambda)$ 
(resp., by $\hat{\cP}^{>}_b(\nu,\lambda)$)
the set of $b$-decreasing paths from 
$\nu$ to $\lambda$ in the graph $\Gamma^{\chi}$ (resp., $\tilde{\Gamma}^{\chi}$).


\subsubsection{}
\begin{cor}{cortail}
Take $\lambda,\nu \in \Lambda^{\chi}$.
\begin{enumerate}
\item
Assume that for each $\mu\in \Lambda^{\chi}$  one has
$$\bigl(\exists i\ \text{ s.t. }\ K^{\mu,\eta}_{\fp^{(i+1)},\fp^{(i)}}(z)\not=0\bigr) \ \Longrightarrow\ \ \eta\in \Lambda^{\chi}.$$
Then
$$K^{\lambda,\nu}_{\fg,\fb}(-1)=\displaystyle\sum_{P\in\hat{\cP}^{>}_b(\nu,\lambda)}
\kappa_{-1}(P).$$
\item
Assume that $\Gamma^{\chi}$ satisfies (Tail) and that 
for each $\mu\in \Lambda^{\chi}$  one has
$$\bigl(\exists i\geq \tail(\mu)\ \text{ s.t. }\ K^{\mu,\eta}_{\fp^{(i+1)},\fp^{(i)}}(z)\not=0\bigr) \ \Longrightarrow\ \ \eta\in \Lambda^{\chi}.$$
Then
$\ K^{\lambda,\nu}_{\fg,\fp_{\lambda}}(-1)=\displaystyle\sum_{P\in {\cP}_b^{>}(\nu,\lambda)}\kappa_{-1}(P).$
\end{enumerate}
\end{cor}
\begin{proof}
The assertions follow from the formula~(\ref{chainK}).
\end{proof}

\subsubsection{Remark}
The  graph $\Gamma^{\chi}$ is useful for character formulae.
Retain notation of~\ref{euler} and notice that
 $\cE_{\lambda,\fp}$ has a particularly nice formula if $\lambda|_{\fl\cap \fh}=0$
(in this case $\ch L_{\fp}(\lambda)=e^{\lambda}$). Thus it makes sense to express
$\ch L(\mu)$ in terms of $\cE_{\lambda,\fp^{(j)}}$ for $j\leq \tail(\lambda)$.
By~\ref{remE}, $\cE_{\lambda,\fp}$ can be zero if ``$\fp$ is too small'';
thus it makes sense to consider the maximal ``nice'' $\fp$ for each $\lambda$, 
which is $\fp_{\lambda}$.

\subsubsection{}
\begin{lem}{lem0mindom}
If the zero weight is a minimal
dominant weight for $\fl^{(p)}$ for each $p$, then $\hat{\Gamma}^{\chi}(0)={\Gamma}^{\chi}(0)$.
\end{lem}
\begin{proof}
 By~\ref{Gamma0}, as a $\fl^{(p)}$-module 
$\Gamma^0_{\fp^{(p)},\fp^{(p-1)}}(L_{\fp^{(p-1)}}(\lambda))$ 
is a finite-dimensional quotient of the Verma $\fl^{(p)}$-module
$M_{\fl^{(p)}}(\lambda)$. If
the zero weight is a minimal
dominant weight for $\fl^{(p)}$, then  for each 
$p\leq \tail(\lambda))$ the module
$L_{\fp^{(p)}}(\lambda)$
is a unique finite-dimensional quotient of
$M_{\fl^{(p)}}(\lambda)$, so $\Gamma^0_{\fp^{(p)},\fp^{(p-1)}}(L_{\fp^{(p-1)}}(\lambda))=L_{\fp^{(p)}}(\lambda)$. Therefore
$\hat{\Gamma}^{\chi}(0)={\Gamma}^{\chi}(0)$ as required.
\end{proof}

\subsubsection{Definitions}\label{bipartition}
Let $\pari: V\to \mathbb{Z}_2=\{0,1\}$ be any map.

We say that
$\Gamma^{\chi}(0)$ (resp., of $\hat{\Gamma}^{\chi}(0)$) 
is bipartite {\em with respect to $\pari$}
if for each edge $\nu\overset{e}{\longrightarrow} \lambda$ in this graph $\pari(\lambda)\not=\pari(\nu)$.

Recall that  $\kappa(e)$ is a polynomial with non-negative integral
coefficients. 
We say that $\pari$ gives a {\em parametric bipartition} of 
$(\Gamma^{\chi},\kappa)$  if for each
edge $\nu \overset{e}{\longrightarrow}\lambda$
 one has 
\begin{equation}\label{paridefeq}
z^{\pari(\lambda)-\pari(\nu)+1} \kappa(e)\in \mathbb{Z}[z^2].\end{equation}

We say that $\pari$  gives a {\em signed bipartition} of 
$(\Gamma^{\chi}(-1),\kappa_{-1})$  if for each edge $\nu\overset{e}{\longrightarrow} \lambda$ 
\begin{equation}\label{dexicc}
(-1)^{\pari(\lambda)-\pari(\nu)+1}\kappa_{-1}(e)\in\mathbb{Z}_{\geq 0}
\end{equation}
or, equivalently, 
\begin{equation}\label{dexic}
(-1)^{\pari(\lambda)-\pari(\nu)}(-1)^{length\, (P)}\kappa_{-1}(P)
\in\mathbb{Z}_{\geq 0}\ \ \text{ for each path $P$ from $\nu$ to $\lambda$}.\end{equation}

We say that $\pari$ gives a {\em parametric bipartition} of 
 $(\hat{\Gamma}^{\chi},\kappa)$ if each
edge $\nu \overset{e}{\longrightarrow}\lambda$ with $\nu\not=\lambda$ satisfies~(\ref{paridefeq}). Similarly, we say that $\pari$  gives a {\em signed bipartition} of 
 of $(\hat{\Gamma}^{\chi}(-1),\kappa_{-1})$ if~(\ref{dexicc}) holds for 
 each $\nu \overset{e}{\longrightarrow}\lambda$ with $\nu\not=\lambda$
 or, equivalently if~(\ref{dexic}) holds for any paths without loops.

\subsubsection{}
Recall that $\hat{\Gamma}^{\chi}(0)$  does not have loops.
If $\pari$  is a parametric bipartition of 
$\Gamma^{\chi}$ (resp., $\hat{\Gamma}^{\chi}$), then 
$\pari$ 
is  a signed bipartition of 
$(\Gamma^{\chi}(-1),\kappa_{-1})$ (resp., of $(\hat{\Gamma}^{\chi}(-1),\kappa_{-1})$)
and  $\Gamma^{\chi}(0)$ (resp.,  $\hat{\Gamma}^{\chi}$)
is bipartite  with respect to $\pari$.

\subsubsection{Remark}
In the examples~\ref{exampr}--\ref{G3F4} below  $\Gamma^{\chi}$ is a $\mathbb{Z}$-graded graph admitting
a parametric bipartition; the same is true
for the graph $\hat{\Gamma}^{\chi}$ for the dense flag for a distinguished
Borel in $\fgl(n|n)$-case, see~\cite{MS}. The graph $\hat{\Gamma}^{\chi}$ has a loop
$0 \overset{e}{\longrightarrow} 0$ 
(and thus is not $\mathbb{Z}$-graded) for the dense flag for a mixed
Borel for 
$\fg=\osp(2|2)$ ($=\fsl(1|2)$) and for $\fg=\fosp(4|2)$ see~\ref{osp22} and~\ref{G3F4}.
By~\cite{GS}, Lemma 26, 
 $\kappa(e)=z$ for $\osp(2|2)$ and $\kappa(e)=z^2$ for $\osp(4|2)$ 
 so the formula~(\ref{paridefeq}) holds for $e$ if $\fg=\osp(2|2)$ and does not hold
if $\fg=\osp(4|2)$. The graph $\hat{\Gamma}^{\chi}$ admits a  parametric bipartition
in both cases.

\subsection{Graph $\Ext(\chi)$}
By~\ref{remKM}   $\dim\Ext^1(L(\lambda),L(\nu))=\dim\Ext^1(L(\nu),L(\lambda))$.
We denote by $\Ext(\chi)$ the graph without loops, with the set of vertices
$\Lambda^{\chi}$ and $\dim\Ext^1(L(\lambda),L(\nu))$  edges between $\nu$ and $\lambda$ for $\nu\not=\lambda$ (we will usually consider the undirected
edges).

\subsubsection{}
We say that $\Ext(\chi)$ is a subgraph of a directed graph  if $\Ext(\chi)$ 
 is a subgraph of the ``undirected version'' of this graph (we forget
the directions of edges).

\subsubsection{}\label{cornongraded}
Recall that 
$\Gamma^0_{\fp^{(p)},\fp^{(p-1)}}(L_{\fp^{(p-1)}}(\lambda))$ is the maximal
finite-dimensional quotient of $\Ind^{\fp^{(p)}}_{\fp^{(p-1)}}(L_{\fp^{(p-1)}}(\lambda))$; this is indecomposable module
with the cosocle  is isomorphic to  $L_{\fp^{(p)}}(\lambda)$.
Using~\Cor{corExtchain} we obtain the

\begin{cor}{cormainold}
The graph $\Ext(\chi)$ is a subgraph of $\hat{\Gamma}^{\chi}(0)$.
\end{cor}

\subsubsection{}
\begin{cor}{maincor}
Assume that $\hat{\Gamma}^{\chi}$ admits a  parametric bipartition $\pari$.
\begin{enumerate}
\item $\Ext(\chi)$ is bipartite with respect to $\pari$.
\item 
$\Gamma^i_{\fp^{(p)},\fp^{(p-1)}}(L_{\fp^{(p-1)}}(\lambda))$ is a semisimple 
$\fp^{(p)}$-module for $i>0$ and has a semisimple radical for $i=0$;

\item
Assume that ${\Gamma}^{\chi}$ admits a parametric bipartition  $\pari$ and  that  the zero weight is a minimal
dominant weight for $\fl^{(p)}$ for each $p=0,1,\ldots,k-1$. Then
$\Ext(\chi)$ is a subgraph of ${\Gamma}^{\chi}(0)$ and $\pari$ defines a bipartition of  $\Ext(\chi)$. Moreover, the claims of  (ii) hold for  $p>\tail(\lambda)$.
\end{enumerate}
\end{cor}
\begin{proof}
\Cor{cormainold} implies (i). For (ii) let
 $\Ext_{(p)}$ be the ``$\Ext$''-graph for $\fp^{(p)}$: the set of vertices
for this graph is $\Lambda^{\chi}$ and the multiplicity of
the edge $\lambda\frac{\ \ \ }{\ \ }\nu$ is 
$\dim\Ext^1(L_{\fp^{(p)}}(\lambda),L_{\fp^{(p)}}(\nu))$. By~\Cor{corExtchain},
 $\Ext_{(p)}$ is a subgraph of $\hat{\Gamma}^{\chi}(0)$, so
 $\pari$ gives a bipartition on $\Ext_{(p)}$.
 For $i>0$ one has
 $$[\Gamma^i_{\fp^{(p)},\fp^{(p-1)}}(L_{\fp^{(p-1)}}(\lambda)):L_{\fp^{(p)}}(\nu)]\not=0\ \ \Longrightarrow \ \ \pari(\nu)+\pari(\lambda)\equiv i+1\ \mod 2.$$
 Therefore there are no non-splitting extensions between
the subquotients
of the $\fp^{(p)}$-module 
$\Gamma^i_{\fp^{(p)},\fp^{(p-1)}}(L_{\fp^{(p-1)}}(\lambda))$; thus 
 this module is completely reducible. For $i=0$
the same holds for $\nu\not=\lambda$ and $L_{\fp^{(p)}}(\lambda)$
is a unique simple quotient of $\Gamma^0_{\fp^{(p)},\fp^{(p-1)}}(L_{\fp^{(p-1)}}(\lambda))$. Hence the radical of
$\Gamma^0_{\fp^{(p)},\fp^{(p-1)}}(L_{\fp^{(p-1)}}(\lambda))$ is semisimple. This gives (ii). 
If the zero weight is a minimal
dominant weight for $\fl^{(p)}$ for each $p$, then
$\tilde{\Gamma}^{\chi}(0)={\Gamma}^{\chi}(0)$ (see~\Lem{lem0mindom}) and so (iii) has the same proof as (ii).
\end{proof}

\subsubsection{Remark}\label{exapr}
We see  that
in order to have a parametric partition on $\hat{\Gamma}^{\chi}$ one has to take
a ``dense enough'' chain of the parabolic subalgebras, since
if $\hat{\Gamma}^{\chi}$ admits such grading, then the
 maximal finite-dimensional quotient of
$\Ind^{\fp^{(p)}}_{\fp^{(p-1)}}(L_{\fp^{(p-1)}}(\lambda))$ 
has a Loewy filtration of length $\leq 2$. In the examples 
below we take $\fl^{(p)}$ of the defect $p$.

\subsection{The Gruson-Serganova algorithm}\label{twographs}
We assume that ${\Gamma}^{\chi}=(\Lambda^{\chi}, E)$  is a $\mathbb{Z}$-graded graph which satisfies  the assumptions of ~\Cor{cortail} (ii).
The following construction is a slight reformulation of the construction described in~\cite{GS}, Sect. 12.

\subsubsection{}
Recall that $\ch L_{\fp_{\lambda}}(\lambda)=e^{\lambda}$.
Set
\begin{equation}\label{eqE}
\mathcal{E}_{\nu}:=\mathcal{E}_{\nu,\fp_{\nu}}=
 R^{-1}e^{-\rho}
\displaystyle\sum_{w\in W} \sgn(w) w\bigl(\frac{e^{\nu+\rho}}{\displaystyle\prod_{\alpha\in\Delta(\fl_{\nu})_1^+}
(1+e^{-\alpha})}\bigr).\end{equation}
Combining~\ref{euler} and~\Cor{cortail} (ii) we get 
$$\sum_{\mu} a^{>}_{\lambda,\mu}  \ch L(\mu)=\mathcal{E}_{\lambda},$$
for $A^{>}(\kappa_{-1})=(a^{>}_{\lambda,\nu})$ defined as in~\ref{Finass}. The matrix $A:=A^{>}(\kappa_{-1})$ is lower-triangular
with $a_{\lambda,\lambda}=1$.
 Thus $A$ is invertible that is  
$$\ch L(\lambda)=\sum_{\mu} a'_{\lambda,\mu}\mathcal{E}_{\mu},$$
for  $(a'_{\lambda,\mu}):=A^{-1}$.
In the light of~\Lem{lemAA} the entries of $A^{-1}$ can be expressed in terms of  $b$-increasing paths from $\nu$ to $\lambda$  if $b: E\to\mathbb{Z}$
{\em would satisfy the  property} (BB). 
Unfortunately,  $b$ almost never satisfy (BB);
however,  it is often possible to find a decreasingly-equivalent function $b'$
satisfying (BB) (we do not require that $b'$  satisfies (Tail)).
 For  $\fgl(M|N)$ and $\osp(M|N)$  the function $b'$ is given
in~\cite{MS} and~\cite{GS} respectively; in~\ref{Dgosp} below
we  describe $b'$ for $\osp(M|N)$-case. Denoting by $\cP^{<}_{b'}(\nu,\lambda)$
the set of $b'$-increasing paths in $\Gamma^{\chi}$  we obtain
$$a^{>}_{\lambda,\mu}=\sum_{P\in \cP^{>}_{b}(\mu,\lambda)}\kappa_{-1}(P)=\!\!\sum_{P\in \cP^{>}_{b'}(\mu,\lambda)}\kappa_{-1}(P),\ \ \ 
a'_{\lambda,\mu}=\sum_{P\in \cP^{<}_{b'}(\mu,\lambda)}(-1)^{length(P)}\kappa_{-1}(P).$$

\subsubsection{}
Assume that $\pari: V\to \{0,1\}$ is a signed bipartition
of $(\Gamma^{\chi}(-1),\kappa_{-1})$ (see~\ref{bipartition} for definition).
By~(\ref{dexic})  the number
$(-1)^{\pari(\lambda)-\pari(\mu)}a'_{\lambda,\mu}$ is a non-negative integer.
(i.e. the Gruson-Serganova character
formula is ``positive'').   These number can be interpreted as follows.
Consider the following modification of the graph $\Gamma^{\chi}=(V,E)$: the  graph $D^{\chi}$  with the same set of vertices $V=\Lambda^{\chi}$ and the set of egdes $E'$ obtained from $E$ by 
taking each edge $\nu\overset{e}{\longrightarrow}\lambda$ with the multiplicity $(-1)^{\pari(\lambda)-\pari(\nu)}\kappa_{-1}(e)$
(this number is non-negative since $\pari$ is a signed bipartition). By above,
$$(-1)^{\pari(\lambda)-\pari(\mu)}a'_{\lambda,\mu}\text{ is the number of $b'$-increasing paths from $\mu$ to $\lambda$ in $D^{\chi}$.}$$

For $\osp(M|N)$-case the graph $D^{\chi}$ is described in~\cite{GS}; 
we give some details in~\ref{GScharosp} below;
the case $\fgl(m|n)$ will be treated in~\cite{GH2}.

\subsubsection{}\label{trick}
The assumption that $\Gamma^{\chi}$ is $\mathbb{Z}$-graded
can be weaken using the following trick. Fix a set of ``bad vertices''
$\Lambda'\subset\Lambda(\chi)$ and
consider a graph $\Gamma'(\chi)$ obtained from $\Gamma^{\chi}$ by erasing
all edges ending at $\lambda\in\Lambda'$. Assume that 
$\Gamma'(\chi)$ is  $\mathbb{Z}$-graded.
The above reasoning allows to express $\ch L(\mu)$ in terms of 
$\mathcal{E}_{\lambda}$ for $\lambda\in\Lambda(\chi)\setminus \Lambda'$ and
$\ch L(\nu)$ for $\nu\in \Lambda'$, see~\ref{G3F4} for examples.

\subsection{Examples}
The Poincar\'e polynomials for certain  chains of parabolic subalgebras were
computed for the finite-dimensional Kac-Moody superalgebras 
in~\cite{S},~\cite{MS},~\cite{Germoni2},\cite{Lilit}
and for $\fq_n$ in~\cite{PS1}.
 In all these cases the chain satisfies the following condition:
 $\fl^{(p)}$ has defect $p$. Below we list some properties of the corresponding graphs (for the $\fgl$-case we consider only the principal block in $\fgl(n|n)$).  

In all these examples
the Poincar\'e polynomials have the following property: the polynomial
$K^{\lambda,\mu}_{\fp^{(s)}, \fp^{(s-1)}}-\delta_{\lambda,\mu}$ is non-zero for at most one value of  $s> \tail(\lambda)$. We denote by $\kappa^{\lambda,\mu}$
the corresponding non-zero polynomial (if it exists) and set $\kappa^{\lambda,\mu}=0$ otherwise.

The above property implies that $\Gamma^{\chi}$ does not have multi-edges
(and that $\kappa(\mu\to\lambda)=\kappa^{\lambda,\mu}$). 
The graph $\Gamma^{\chi}$  admits a $\mathbb{N}$-grading and satisfies (Tail).

 For $\fg\not=\fq_n$ the graph
$\Gamma^{\chi}$  admits 
a parametric bipartition $\pari$. 

For $\fg\not=\fq_n,\osp(2m|2n)$, the polynomials $\kappa^{\lambda,\mu}$ 
 are monomials,  so the condition
on the parametric bipartition simply means that $\kappa^{\lambda,\mu}$ is zero or
$z^i$ for $i\equiv \pari(\lambda)-\pari(\mu)+1$ modulo $2$.
In these cases $D^{\chi}=\Gamma^{\chi}$. 

For $\fg=\osp(2m|2n),\fq_n$ with $\chi$ of atypicality greater than one,  
 $\kappa^{\lambda,\mu}\in \{0,z^i, z^i+z^{j}\}$
and the condition on  the parametric bipartition takes the  form
$i\equiv \pari(\lambda)-\pari(\mu)+1$  and $i\equiv j$ modulo $2$. This holds for
$\osp(2m|2n)$ in this case $D^{\chi}$ is obtained from $\Gamma^{\chi}$
by doubling the  edges with $\kappa^{\lambda,\mu}=z^i+z^{j}$.

An interesting example is the $\fq_n$-case. For $\chi$ of atypicality greater than one, $i-j$ can be odd, so 
$\Gamma^{\chi}$ does not admit
a parametric bipartition  and $\Gamma^{\chi}(0)$ is not bipartited.
By~\cite{MM} the $\Ext$-graph is not bipartite. However, 
the $(\Gamma^{\chi}(-1),\kappa_{-1})$ admits a signed bipartition
$\pari$.
In this case $D^{\chi}$ is obtained from $\Gamma^{\chi}$
by doubling the edges with even $j$ and deleting the edges with 
odd $j$ ($\kappa^{\lambda,\mu}(-1)=0$ if $j$ is odd). The  Gruson-Serganova algorithm gives the Su-Zhang character formula~\cite{SZq}.

\subsection{Examples of defect one}\label{exampr}
We start form the examples when $\fg$ has defect $1$ and
the chain is $\fb=\fp^{(0)}\subset\fp^{(1)}=\fg$.
 
In this case $\ b(e)=1$ for each $e\in \hat{\Gamma}^{\chi}$,
so the conidtion (BB) does not hold if $\Gamma^{\chi}$ contains
paths of length two. Moreover, the descreasing paths in $\Gamma^{\chi}$
are the paths containing at most one edge. Thus  any function $b': E\to \mathbb{Z}$ 
satisfying $b'(e_1)<b'(e_2)$ for each path 
$\cdot\overset{e_1}{\longrightarrow}\cdot\overset{e_1}{\longrightarrow}\cdot$
in $\Gamma^{\chi}$  is descreasingly equivalent to 
$b$ and satisfies (BB).

 We will depict an edge $e$  in 
$\Gamma^{\chi}$ (resp., in $\hat{\Gamma}^{\chi}$) 
as $\nu\overset{j;\kappa(e)}{\longrightarrow}\lambda$ (resp., as 
$\nu\overset{\kappa(e)}{\longrightarrow}\lambda$)
with $j=b'(e)$.

Except
for $\fgl(1|1)$ one has $\tail \lambda=\delta_{0,\lambda}$; for 
$\fgl(1|1)$ one has $\tail\lambda=1$
for each atypical weight $\lambda$. Recall that the condition (Tail) in this case
is tautological.

Except for $F(4), G(3)$  we take $\chi$ corresponding to the principal block
(i.e., $\chi$ is the central character of the trivial module $L(0)$).

In all these  examples the Poincar\'e polynomials $\kappa(e)\in\{1,z\}$; 
the graph 
$\Gamma^{\chi}$ admits a parametric partition $\pari$ (which means that
$\kappa(e)=1$ if $e$ connects the vertices with different value of $\pari$
and $\kappa(e)=z$ otherwise). In particular, $D^{\chi}=\Gamma^{\chi}$.

\subsubsection{Example: $\fgl(1|1)$}\label{gl11}
We take $\fg:=\fgl(1|1)$ with $\Delta^+=\{\alpha\}$.

The simple modules in the principal block are $\{L(s\alpha)| s\in\mathbb{Z}\}$.
As a vector space $L(s\alpha)\cong \Pi^s(\mathbb{C})$, 
so $\sdim L(s\alpha)=(-1)^s$; we define
$\pari(L(s\alpha)):=s$ modulo $2$.
For a non-zero $x\in\fg_{1}$ one has
$\DS_x(L(s\alpha))=\Pi^s(\mathbb{C})$. We have

$$\begin{array}{lcccccccccc}
\hat{\Gamma}^{\chi} & &
\ldots  &\overset{(1)}{\longrightarrow}&-\alpha &\overset{(1)}{\longrightarrow}& 0  &\overset{(1)}{\longrightarrow}&\alpha&\overset{(1)}{\longrightarrow}&\ldots\\
\Gamma^{\chi} & & \ldots & &-\alpha & & 0& & \alpha & &\ldots\\
\Ext(\chi) & & \ldots  &\longleftrightarrow&-\alpha &\longleftrightarrow& 0 &\longleftrightarrow&\alpha &  \longleftrightarrow &     \ldots\\
\end{array}$$
Clearly, $\pari(L(s\alpha)):=s$  modulo $2$ defines
 a parametric partition on $\hat{\Gamma}^{\chi}$ and a bipartition
on $\Ext(\chi)$; the graph $\hat{\Gamma}^{\chi}$ does not admit
a $\mathbb{N}$-grading.

\subsubsection{Example: $\osp(2|2)$}\label{osp22}
Take $\fg=\osp(2|2)\cong \fsl(2|1)$ with the base $\Sigma=\{\delta_1\pm\vareps_1\}$.

The simple modules in the principal block are $\{L(\lambda_s)| s\in\mathbb{Z}\}$,
where $\lambda_s:=|s|\delta_1+s\vareps_1$. Note that $0$ is a minimal
dominant weight, so the assumption of~\Cor{maincor} (iii) holds.
The extension graph $\Ext(\chi)$  is
$A_{\infty}^{\infty}$. The Poincar\'e polynomials $K^{\lambda,\nu}_{\fg,\fb}$ were computed in~\cite{GS}, Sect. 12.   One has
$$\begin{array}{lcccccccccc}
{\Gamma}^{\chi} & &
\ldots  &\overset{(2;1)}{\longleftarrow}&\lambda_{-1} &\overset{(1;1)}{\longleftarrow}& \lambda_0  &\overset{(1;1)}{\longrightarrow}&\lambda_1&\overset{(2;1)}{\longrightarrow}&
\ldots\\
\Ext(\chi)=A_{\infty}^{\infty} & & \ldots  &\longleftrightarrow&\lambda_{-1} &\longleftrightarrow& \lambda_0 &\longleftrightarrow&\lambda_1 &  \longleftrightarrow &     \ldots\\
\end{array}$$
By~\cite{GS}, Lemmata 25, 26 
the graph $\hat{\Gamma}^{\chi}$ 
can be obtained from  the graph
${\Gamma}^{\chi}$ by adding a loop around $\lambda_0$ which is marked by $(1;z)$.
Observe that $\pari(\lambda_j):=p(\lambda)\equiv j$ module $2$
is a parametric partition on $\hat{\Gamma}^{\chi}$ and is a bipartition
on $\Ext(\chi)$.
The function $||\lambda_i||_{gr}:=|i|$ gives a $\mathbb{N}$-grading on
${\Gamma}^{\chi}$; the graph $\hat{\Gamma}^{\chi}$ does not admit
a $\mathbb{N}$-grading (since it has a loop).

The Gruson-Serganova formula is
\begin{equation}\label{eqchar22}
\ch L(\lambda_{\pm j})=\sum_{s=0}^j (-1)^{j-s}\mathcal{E}_{\lambda_{\pm s}}\ \ \ 
\text{ for }j\geq 0.
\end{equation}

\subsubsection{Example: $\osp(3|2)$}\label{osp32}
Take $\fg=\osp(3|2)$ with the base
 $\Sigma=\{\vareps_1-\delta_1,\delta_1\}$.
The simple modules in the principal block are $\{L(\lambda_s)\}_{s=0}^{\infty}$,
where $\lambda_0:=0$ and $\lambda_s:=(s-1)\delta_1+s\vareps_1$ for $s>1$;
$0$ is a minimal dominant weight. The Poincar\'e polynomials 
$K^{\lambda,\nu}_{\fg,\fb}$ were computed in~\cite{Germoni2}. One has
$$  
\xymatrix{&\hat{\Gamma}^{\chi}: &\lambda_1\ar^{(1)}[r]\ar[rd]
&\lambda_2\ar^{(1)}[r] &\lambda_3
\ar^{(1)}[r]&\ldots& \\
& && \lambda_0\ar^{(z)}[lu]\ar_{(1)}[u]&&&
}
$$
The map $\pari(\lambda_0):=0$, $\pari(\lambda_i):\equiv i-1$ modulo $2$ for $i\not=0$ 
 is a parametric partition on $\hat{\Gamma}^{\chi}$ (one has $\pari(\lambda_i)=p(\lambda_i)$).  The graph $\Gamma^{\chi}$ is
$$  
\xymatrix{& &\lambda_1\ar^{(2;1)}[r]&\lambda_2\ar^{(3;1)}[r] &\lambda_3
\ar^{(4;1)}[r]&\ldots& \\
& && \lambda_0\ar^{(1;z)}[lu]\ar_{(2;1)}[u]&&&
}
$$
and the ``undirected version'' of $\Ext(\chi)$ coincides with $\Gamma^{\chi}(0)$.
The function $||\lambda_i||_{gr}:=i$ gives
an  $\mathbb{N}$-grading $\Gamma^{\chi}$ (note that
the graph  $\hat{\Gamma}^{\chi}$
is not $\mathbb{Z}$-graded). 

The Gruson-Serganova  formula is
$\ch L(\lambda_0)=\mathcal{E}_{0}=1$ and
\begin{equation}\label{eqchar32}
 \ch L(\lambda_1)=\mathcal{E}_{\lambda_{0}}+\mathcal{E}_{\lambda_{1}},\ \ \ 
\ch L(\lambda_{j})= 2(-1)^{j-1} \mathcal{E}_{\lambda_{0}}+\sum_{s=1}^j (-1)^{j-s} \mathcal{E}_{\lambda_{s}}\ \ \text{ for }j>1.
\end{equation}

%
%
%

\subsubsection{Remark}\label{rem2221}
For $\osp(3|2)$ we have two bases: the ``mixed'' base
 $\{\vareps_1-\delta_1;\delta_1\}$
and the base $\{\delta_1-\vareps_1;\vareps_1\}$. The computations in~\cite{Germoni2}
are performed for the second base; it is not hard to see that
for the first base  the results are the same.

For  $\osp(2|2)\cong \fsl(1|2)$ we have two bases: the ``mixed'' base $\{\delta_1\pm\vareps_1\}$ and the distinguished base
$\{\vareps_1-\delta_1;2\delta_1\}$.  The graphs for the mixed base
are given in~\ref{osp22}; the graphs for the distinguished base are the same
as in~\ref{gl11} (notice that $0$ is a minimal dominant weight 
for the mixed base, whereas for the distinguished base
the set of dominant weights does not have minimal elements).

\subsubsection{Cases $\osp(4|2)$, $G(3)$ and $F(4)$}\label{G3F4} 
 Recall that
$\Ext(\chi)$ is either $D_{\infty}$ or $A^{\infty}_{\infty}$, see~\ref{DSat1}.
The  graphs $\hat{\Gamma}(\chi)$ for a certain distinguished  Borel subalgebras
were computed in~\cite{Germoni2}, \cite{Lilit}: this graph is 
 is the same as
for $\osp(3|2)$ (resp., as for $\osp(2|2)$) if $\Ext(\chi)=D_{\infty}$ and 
(resp.,  if $\Ext(\chi)=A^{\infty}_{\infty}$).

For the principal blocks the graph $\Ext^1(\chi)$ is  $D_{\infty}$
and so the graph $\hat{\Gamma}(\chi)$  is 
the same as for $\osp(3|2)$.
Since $\tail 0=1$ and  $\tail\lambda=0$ for each $\lambda\not=0$,
the graph $\Gamma^{\chi}$ is  the same as  for $\osp(3|2)$. 

For the non-principal blocks   one has
$\hat{\Gamma}^{\chi}=\Gamma^{\chi}$ (since $\tail\lambda=0$ for each $\lambda\in\Lambda^{\chi}$).

The character formulae for these cases were obtained in~\cite{Germoni2},\cite{Lilit}.
The above approach give other type of character formulae.

By above, for the principal block  $\Gamma^{\chi}$ is the same as for $\osp(3|2)$,
so we obtain the same character formula~(\ref{eqchar32}).
Consider a non-principal block. By above, $\Gamma^{\chi}=\hat{\Gamma}^{\chi}$
is the same as the graph $\hat{\Gamma}^{\chi}$ 
for $\osp(2|2)$ or for $\osp(3|2)$. In both cases
${\Gamma}^{\chi}$ have  cycles and all these cycles contain $\lambda_0$;
the graph $\Gamma'$ which is obtained from
$\Gamma^{\chi}$ by erasing all edges ending at $\lambda^0$ is
$\mathbb{N}$-graded. Using~\ref{trick}, we get
Gruson-Serganova type character formulae which can be obtained 
from~(\ref{eqchar22}) and~(\ref{eqchar32})
respectively by changing $\mathcal{E}_0$ by $\ch L(\lambda_0)$
(notice that
$\ch L(\lambda_0)$ is given by the Kac-Wakimoto formula).

\subsection{Remark on Gruson-Serganova type character formulae}\label{remGSchar}
Let $\cB$ one of the blocks  of atypicality $1$ considered in Section~\ref{sectatyp1}.
Then $\Gamma^{\chi}$ is one of the graphs appeared in~\ref{osp22}, \ref{osp32} and
we call $\lambda_0$ a {\em Kostant weight};
notice that  $\lambda_0$ is uniquely defined in terms of $\Gamma^{\chi}$
which has fewer  automorphisms than
$\Ext(\chi)$.

Take $L(\lambda_i)\in \Irr(\cB)$ and write 
 $\ch L(\lambda_i)=\sum_i a_i\mathcal{E}_{\lambda_i}$
using~(\ref{eqchar22}), (\ref{eqchar32}).
By~\ref{DSat1} one has $\DS_x(L(\lambda_i))=\Pi^s(\DS_x(L(\lambda_0)))^{a_0}$
for $s=0$ if $a_0>0$ and  $s=1$ if $a_0<0$. This can be translated to the language 
of supercharacters in the following manner.

Retain notation of~\ref{subsectpari}. Denote by
 $\Sch(\fg)$ the image of the map
$\sch: \tilde{\cF}(\fg)\to \mathbb{Z}[\Lambda_{m|n}]$.
Since $\sch (\Pi(V))=-\sch(V)$ one has
 $\sch(\tilde{\cF}(\fg))=\sch (\cF(\fg))$.
For Lie superalgebras~(\ref{list}) the ring $\Sch(\fg)$ is isomorphic to the reduced
Grothendieck ring of $\tilde{\cF}(\fg)$ and $\DS_x$ induces
an algebra homomorphism $ds'_x:\Sch(\fg)\to \Sch(\fg_x)$
given by $f\mapsto f|_{\fh_x}$, see~\cite{HR}.

For $V\in\cF(\fg)$ one has $\sch V=\pi(\ch V)$, where
$\pi: \mathbb{Z}[\Lambda_{m|n}]\to \mathbb{Z}[\Lambda_{m|n}]$ is the involution $\pi(e^{\mu}):=p(\mu)e^{\mu}$. In particular, $\{\pi(\mathcal{E}_{\lambda})|\ 
\lambda\in\Irr(\cF(\fg)))$ forms a basis of $\Sch(\fg)$. If $\fg$ has defect $1$,
then the kernel of the map $ds'_x:\Sch(\fg)\to \Sch(\fg_x)$ 
is spanned by the basis elements 
 $\pi(\mathcal{E}_{\lambda})$ for
$\lambda$s which are not  Kostant weights. A similar property hold for
the $\fgl(1|n)^{(1)})$-case; in~\cite{GH2} 
we will show that this holds for $\osp(m|n)$-case as well
(the situation is more complicated for $\fgl(m|n)$).

\subsection{Example: $\fg=\fgl(k|k)$}\label{glkk}
Take $\fg=\fgl(k|k)$ 
 with a distinguished Borel subalgebra $\fb$.
For $p=0,\ldots,k$ we denote by $\fp^{(p)}$ a  parabolic subalgebra containing
$\fb$ with the Levi factor $\fl^{(p)}\cong \fgl(p|p)$
($\fp^{(p)}$ is unique since $\fb$ is a distinguished Borel).
We consider  the corresponding chain of parabolic subalgebras~(\ref{parabolicchain}). 

The Poincar\'e polynomials $K^{\lambda,\nu}_{\fp^{(p+1)},\fp^{(p)}}$ were computed
in~\cite{S},~\cite{MS} Cor. 3.8. 
The graph $\hat{\Gamma}^{\chi}$ is $\mathbb{Z}$-graded and
does not have multi-edges. For the principal block  the condition (Tail) holds.
The map $p(\lambda)$ defines a parametric partition on $\hat{\Gamma}^{\chi}$.
By~\cite{MS}, Sect. 6,  $\Ext(\chi)=\hat{\Gamma}^{\chi}(0)$ (this can be also deduced from~\cite{BS}). One has $\Ext(\chi)\not={\Gamma}^{\chi}(0)$ 
(see the example of $\fgl(1|1)$ above).

\section{Case $\osp(M|N)$}\label{sectosp}
In this section $\fg=\osp(M|N)$.   The category $\Fin(\fg)$ was studied in~\cite{GS} and~\cite{ES}. In this section we deduce the existence
of $\pari$ satisfying (Dex1) from the results of~\cite{GS}.
Another approach is developed in~\cite{ES}, \cite{ES2}.
By~\cite{GS}, each block of atypicality $k$ in $\tilde{\cF}(\fg)$ is equivalent either  to a principal $\fosp(2k+t|2k)$-block $\cB$  or to $\Pi(\cB)$,
 where $t=1$ for odd $M$  and $t=0,2$ for even $M$, and
 this equivalence is ``compatible with character formula'',
see~\Rem{genchar}. 

We fix a ``mixed'' base consisting of odd roots, see~\ref{mix} below.
We denote by $\chi$
the central character of the principal block $\cB$ and retain notation of~\ref{usefulgraphs}.
For each $\lambda\in\Lambda^{\chi}$ we set
$\pari(L(\lambda))=p(\lambda)$ if
$t=0,1$;  for $t=2$ we define
$\pari(L(\lambda))$ 
via a one-to-one correspondence between the simple modules
in the principal blocks for $\osp(2n+1|2n)$ and $\osp(2n+2|2n)$, see~\ref{paridef} below.

The multiplicities $^iK^{\lambda,\nu}_{\fp^{(p)},\fp^{(p-1)}}$ were computed
by C.~Gruson and V.~Serganova  in~\cite{GS} (see~\cite{GSsmall} for small rank examples). We will recall their results  and  describe 
the graphs $\Gamma^{\chi}, D^{\chi}$ in~\ref{gammai}. We will see that $\Gamma^{\chi}$ is $\mathbb{N}$-graded and satisfies (Tail). We will check that
$\pari$  is a parametric partition and that $b$ and $b'$ are
decreasingly-equivalent. As a result~\Cor{maincor}  holds for the block $\cB$
and the  character formula
 (15) in~\cite{GS} 
can be rewritten in the form~(\ref{GSchar}).

Everywhere
in this section, except for~\Rem{genchar}, we take  $\fg=\osp(2k+t|2k)$ for $t=0,1,2$. 

\subsection{Notation}\label{hwtpr}
We take $\fg:=\osp(2k+t|2k)$ for $t=0,1,2$.
The  integral weight lattice  $\Lambda_{k+\ell|k}$ is spanned by $\{\vareps_i\}_{i=1}^{k+\ell}\cup\{\delta_i\}_{i=1}^k$, where
$\ell:=0$ for $t=0,1$ and $\ell:=1$ for $t=2$; the parity function 
 is given by 
$p(\vareps_i)=\ol{0}$, $p(\delta_j)=\ol{1}$ for all $i,j$.

\subsubsection{}\label{mix}
We fix a triangular decomposition
 corresponding to the ``mixed'' base:
 $$\Sigma:=\left\{\begin{array}{ll}
\vareps_1-\delta_{1},\delta_1-\vareps_2,\ldots,\vareps_{k}-\delta_{k},
\delta_k\ & \text{ for }\osp(2k+1|2k)\\
\delta_{1}-\vareps_1,\vareps_1-\delta_{2},\ldots,\vareps_{k-1}-\delta_{k},
\delta_k\pm\vareps_k & \text{ for } \osp(2k|2k)\\
\vareps_{1}-\delta_1,\delta_1-\vareps_{2},\ldots,\vareps_{k}-\delta_k,\delta_k\pm\vareps_{k+1} & \text{ for }
\osp(2k+2|2k).\\
\end{array}\right.$$
We have $\rho=0$ for $t=0,2$ and $\rho=\displaystyle\frac{1}{2}\sum_{i=1}^k(\delta_i-\vareps_i)$ for $t=1$.

\subsubsection{}\label{gp}
 We consider the embeddings 
$$\osp(t|0)\subset\osp(2+t|2)\subset\osp(4+t|4)\subset\ldots\subset\osp(2k+t|2k)=\fg$$
where $\osp(2p+t|2p)$ corresponds to the last $2p+\ell$ roots in $\Sigma$; 
we denote the subalgebra $\osp(2p+t|2p)$ by $\fl^{(p)}$. Note that
$\fl^{(k)}=\fg$ and $\fl^{(0)}=0$ for $t=0,1$,
$\fl^{(0)}=\mathbb{C}$ to $t=2$.

For $p=0,\ldots,k$ we consider the parabolic subalgebra $\fp^{(p)}:=\fl^{(p)}+\fb$.
Notice that
 $\fl^{(p)}$ is the Levi subalgebra of $\fp^{(p)}$; as in~\ref{usefulgraphs},
we denote by $\tail(\lambda)$  the maximal index $q$ such that
$\lambda|_{\fh\cap\fl_q}=0$.

%
%
%

\subsection{Highest weights in the principal block}\label{lambda+rho}
For $\lambda\in\Lambda_{k+\ell|k}$ we set
$$a_i:=-(\lambda|\delta_i)$$ 
and notice that
$p(\lambda)=\displaystyle\sum_{i=1}^k a_i$.
By~\cite{GS},  $\lambda\in\Lambda^{\chi}$ if and only if $a_1,\ldots,a_k$
 are non-negative integers with $a_{i+1}>a_i$ or $a_i=a_{i+1}=0$, and
 $$\lambda+\rho=\left\{
\begin{array}{lll}
\displaystyle\sum_{i=1}^{k-1} a_i(\vareps_i+\delta_i)+a_k(\delta_k+\xi \vareps_k) &\text{ for } & t=0\\
\displaystyle\sum_{i=1}^{k} a_i(\vareps_i+\delta_i)&\text{ for } & t=2\\
\displaystyle\sum_{i=1}^{s-1} (a_i+\frac{1}{2})
(\vareps_i+\delta_i)+\frac{1}{2}(\delta_s+\xi  \vareps_s)+\displaystyle\sum_{i=s+1}^{k} \frac{1}{2}
(\delta_i-\vareps_i)&\text{ for } & t=1\end{array}
\right.$$
for $\xi\in\{\pm 1\}$. For $t=1$ we have $1\leq s\leq k+1$
and $a_s=a_{s+1}=\ldots=a_k=0$ if  $s\leq k$ (for $s=k+1$ we have
 $\lambda+\rho=\sum_{i=1}^{k} (a_i+\frac{1}{2})
(\vareps_i+\delta_i)$).

\subsubsection{}\label{wtdiag}
Take $\lambda\in\Lambda^{\chi}$ and define $a_i$ for $i=1,\ldots,k$ as above.
We assign to $\lambda$  a ``weight diagram'', which is
a number line with one or several symbols drawn at each position with non-negative integral coordinate:

we put the sign $\times$ at each  position with the coordinate $a_i$;

for $t=2$ we add $>$ at the zero position;

we add the ``empty symbol'' $\circ$ to all empty positions.

For $t\not=2$ a weight $\lambda\in\Lambda^{\chi}$ is not uniquely determined by the weight diagram constructed by the above procedure. Therefore, for  $t=0$ with $a_k\not=0$ 
and for $t=1$ with $s\leq k$, we write the 
sign of $\xi$ before the diagram  ($+$ if $\xi=1$ and $-$ if $\xi=-1$).

Notice that each position with a non-zero coordinate contains either $\times$ or $\circ$. For $t=0,1$ 
the zero position is occupied either by $\circ$ or by
several symbols $\times$; we
write this as $\times^i$ for $i\geq 0$. Similarly, for
 $t=2$ the zero position is occupied by $\overset{\times^i}{>}$ with $i\geq 0$.

\subsubsection{}
Notice that  $\tail(\lambda)$ is equal to the number of symbols $\times$ at the zero position of the weight diagram for all cases
except when $t=1$ and the diagram has the sign $+$; in the latter case  
the number of symbols $\times$ at the zero position is $\tail(\lambda)+1$.

\subsubsection{Examples}
The weight diagram of $0$ is $\times^k$ for $t=0$, 
$-\times^k$ for $t=1$ and $\overset{\times^k}{>}$ for $t=2$; one has $\tail(0)=k$.

The diagram $+\circ\times\times$ 
corresponds to the $\osp(4|4)$-weight $\lambda=\lambda+\rho=(\vareps_2+\delta_2)+2(\vareps_1+\delta_1)$
  with  $\tail(\lambda)=0$.

The diagram $+\times^3$ 
corresponds to $\osp(7|6)$-weight
 $\lambda=\vareps_1$  with $\tail(\vareps_1)=2$.

The empty diagram correspond to $\osp(0|0)=\osp(1|0)=0$; the diagram
$>$ 
corresponds to the weight $0$  for $\osp(2|0)=\mathbb{C}$.

\subsubsection{}\label{diag}
For $t=0,1,2$ we denote by $Diag_{k;t}$ the set of (signed) weight diagrams.
The above procedure gives a one-to-one correspondence between $\Lambda^{\chi}$
and $Diag_{k;t}$. For each diagram $f\in Diag_{k;t}$ we denote by $\lambda(f)$
the corresponding weight in $\Lambda^{\chi}$. 

In all cases the weight diagrams in $Diag_{k,t}$ contains $k$ symbols $\times$. 

For a  diagram $f$ and  $a\in\mathbb{N}$ 
we denote by $f(a)$ the symbols at the position $a$.
For  $t=0$ (resp., $t=1$) a  diagram in  $Diag_{k,t}$  has a sign
if and only if $f(0)=\circ$ (resp.,  $f(0)\not=\circ$).

\subsubsection{Map $\tau$}\label{tau}   
Following~\cite{GS}, we introduce a bijection
 $\tau: Diag_{k;2}\to Diag_{k;1}$. For $f\in Diag_{k;2}$ the
 diagram $\tau(f)\in Diag_{k;1}$  is constructed by the
following procedure:

we remove $>$ and then
shift all entires at the non-zero positions of $f$ by one position to the left; then we 
 add a sign  in such a way that $\tail(f)=\tail(\tau(f))$:
the sign $+$ if $f(1)=\times$ and
the sign
$-$ if $f(1)=\circ$ and $f(0)\not=>$.
For instance, 
$$\tau(\overset{\times}{>}\circ\times)=-\times\times,\ \ \tau(\overset{\times}{>})=-\times,\ \tau(>\times)=+\times,\ 
\tau(>\circ \times)=\circ \times.
$$

One readily sees that $\tau$ is a one-to-one correspondence.

\subsection{The maps $||\lambda||, ||\lambda||_{gr},\pari$}\label{paridef}
Let $\{a_i\}_1^k$ be the coordinates of the symbols $\times$ in 
a diagram of $\lambda$. We set
$$||\lambda||:=\left\{\begin{array}{ll}
\displaystyle\sum_{i=1}^k a_i & \text{ for } t=0,1\\
||\tau(f)||& \text{ for } t=2, 
\end{array}
\right.\ \ \ ||\lambda||_{gr}:=\left\{\begin{array}{ll}
\displaystyle\sum_{i=1}^k a_i & \text{ for } t=0,2\\
||\tau^{-1}(f)||& \text{ for } t=1
\end{array}
\right.
$$
and
$$\pari(\lambda):\equiv ||\lambda|| \mod 2.$$

Clearly, $||\lambda||,||\lambda||_{gr}\in\mathbb{N}$ and $||\lambda||_{gr}=0$
if and only if $\lambda=0$.

\subsection{Graph $\Gamma^{\chi}$}\label{gammai}
Retain notation of Section~\ref{Gamma}. 
Consider the  chain of parabolic subalgebras~(\ref{parabolicchain}) 
with $\fp^{(p)}$, $p=0,1,\ldots,k$  defined in~\ref{gp}.

The  Poincar\'e polynomials  $K^{\lambda,\nu}_{\fp^{(p)},\fp^{(p-1)}}(z)$ for $p>\tail\lambda$ were computed
in~\cite{GS},  Sect. 11.  It is proven that   the map
$\tau: Diag_{k;2}\iso \ Diag_{k;1}$
(see~\ref{tau}) preserves these polynomials 
(i.e., $K^{\tau(\lambda),\tau(\nu)}_{\fp^{(p)},\fp^{(p-1)}}=K^{\lambda,\nu}_{\fp^{(p)},\fp^{(p-1)}}$); the coefficients
$^iK^{\lambda,\nu}_{\fp^{(p)},\fp^{(p-1)}}$ are $0$ or $1$ and for $t=0,2$ 
one has $^iK^{\lambda,\nu}_{\fp^{(p)},\fp^{(p-1)}}=1$
if and only if the diagram of $\lambda$ can be obtained from the diagram of $\nu$
by a ``move'' of degree $i$ which ends at the $p$th symbol $\times$  
in the diagram of $\lambda$; we will give some details in~\ref{move0} below
and give a descrip[tion in terms of ``arch diagrams'' in Section~\ref{arcs}.

\subsubsection{Moves for $t=0,2$}\label{move0}
Consider the cases $t=0,2$.  Take $f\in Diag_{k;t}$. 
For each $p,q\in\mathbb{N}$ denote by 
$l_f(p,q)$ the number of symbols $\times$ minus the number of symbols $\circ$
strictly between the positions $p$ and $q$ in $f$. 

A diagram  $f\in Diag_{k;t}$ can be transformed to a diagram $g$ by a ``move'' of degree $d$ if
$f$ satisfies certain conditions, and
$g$ is obtained from $f$ by moving either one symbol
$\times$ from a position $p$ to an 
empty position $q$ with $q>p$ or moving two symbols $\times$ from the zero position to empty positions $p,q$ with $p<q$. If $f$ has a sign,
then $g$ has the same sign.
In both cases we say that the move ``ends at the position $q$''.
We will not specify all conditions on $f$, but notice that 
these conditions depend only on $l_f(s,q)$ for $s<q$.

By above, $\tail(f)-\tail(g)$ is $0,1$ or $2$;
the degree  $d$ satisfies the formula 
\begin{equation}\label{formulaford}
d=\left\{\begin{array}{ll}
l_f(p,q) & \text{ if } \tail(f)-\tail(g)\not=1\\
l_f(p,q)\ \text{ or }\ 2\tail(g)+l_f(p,q)& \text{ if } t=0,\ \ \tail(f)-\tail(g)=1\\
2\tail(g)+l_f(p,q)+1& \text{ if } t=2,\ \ \tail(f)-\tail(g)=1,
\end{array}
\right.
\end{equation}
The conditions on $f$ imply that
$d\geq 0$. Except for the case $t=0$ with $\tail(f)-\tail(g)=1$,
$g$ can be obtained from $f$ by
at most one move; for $t=0$  
it is possible sometimes to obtain $g$ from $f$ 
by two moves of different degrees.
We give below examples of several moves and their degrees
$$
\begin{array}{lccl}
\ldots \times \circ \ldots \longrightarrow \ldots \circ \times \ldots
& & & d=0\\
\ldots\times\times\times\circ\circ\ldots \longrightarrow \ldots \circ \times \times \circ\times\ldots
& & & d=1\\
\times\circ \ldots\ \longrightarrow (\pm)\circ\times\ldots & & & d=0\\
\times^2\circ \ldots\ \longrightarrow \times\times\ldots & & & d=0,2.\\
\end{array}$$

\subsubsection{}
\begin{cor}{coremovedeg}
Take $t=0$ or $t=2$.
Let $\nu,\lambda\in\Lambda^{\chi}$ be two weights with the diagrams
$f$ and $g$ respectively. Assume that 
 $g$ is obtained from $f$ by a move of degree $d$.
Then

(i) $||\lambda||_{gr}>||\nu ||_{gr}$ and $\lambda>\nu$;

(ii) if the move ends in the $i$th symbol $\times$ in $g$, then 
$\tail\nu\leq i$;

(ii) $\pari(\lambda)-\pari(\nu)+d\equiv 1\ mod 2$.
\end{cor}
\begin{proof}
The  inequality  $||\lambda||_{gr}>||\nu||_{gr}$
follows from the fact that $\tau$ preserves $||\ ||_{gr}$ and that 
we  move
symbol(s) $\times$  to the right. The  inequality $\lambda>\nu$
follows  from the fact  that we  move
the symbol(s) $\times$  to the right;  (ii) is obvious.
 For (iii) retain notation of~\ref{move0}
and observe that  
$$l_f(p,q)\equiv q-p+1,\ \ \ \pari(\lambda)-\pari(\nu)\equiv ||\lambda||-||\nu|| \mod 2.$$
For $t=0$ one has $d\equiv l_f(p,q)$  by~(\ref{formulaford}), and
$||\lambda||-||\nu||=q-p$ if $\tail(g)-\tail(g)\not=2$ and $q+p$ otherwise.
For $t=2$ the formula~(\ref{formulaford}) gives
$d\equiv l_f(p,q)+\tail(\lambda)-\tail(\nu)$; in this case
$||\lambda||-||\nu||\equiv q-p +\tail(\lambda)-\tail(\nu)$.
This gives (iii).
\end{proof}

\subsubsection{}\label{Dgosp}
Retain notation of~\ref{usefulgraphs}.
For $t=0,2$ the graph $\Gamma^{\chi}=(\Lambda^{\chi},E)$ has   the edges $\nu\overset{e}{\to} \lambda$ with
$b(e)=j$  if and only if 
the diagram of $\lambda$ can be obtained from the diagram of $\nu$
by a move  which ends at the $j$th symbol $\times$  
in the diagram of $\lambda$; in this case we
denote by $b'(e)=q$ the coordinate of the $j$th symbol
$\times$ in the diagram of $\lambda$. For the edge as above the Poincar\'e polynomial is the sum of $z^d$ for all $d$ 
such that $\lambda$ can be obtained from the diagram of $\nu$
by a move of degree $d$. By~(\ref{formulaford})
 $\kappa(e)=z^d$ except for the
case $t=0$ with $\tail(\nu)-\tail(\lambda)=1$; in the latter case
$\kappa(e)=z^d$ or $z^d(1+z^{2\tail(\lambda)})$ (see Section~\ref{arcs} for details). 

By above, $\tau$ gives a bijection between the graphs $\Gamma^{\chi}$
for $t=2$ and $t=1$ and this bijection is compatible with the functions
$b$ and $\kappa$. For $t=1$ we define $b'$ on $\Gamma^{\chi}$
using this bijection.

\subsubsection{}
\begin{cor}{corpa}
\begin{enumerate}
\item  The map $\pari(\lambda)$ is a parametric bipartition on
 $(\Gamma^{\chi},\kappa)$.

\item If $\nu\to\lambda$ is an edge in $\Gamma^{\chi}$, then
$\nu<\lambda$ and $||\nu||_{gr}<||\lambda||_{gr}$. In particular, 
$||\lambda||_{gr}$
defines a  $\mathbb{N}$-grading on
$\Gamma^{\chi}$.

\item The graph $\Gamma^{\chi}$ satisfies 
the assumption  (Tail).

\item The functions $b,b'$  are decreasingly equivalent and $b'$ satisfies 
the property (BB) of~\Lem{lemAA}.
\end{enumerate}
\end{cor}
\begin{proof}
Consider the cases $t=0,2$.
\Cor{coremovedeg} implies (i)--(iii). For (iv) take a path
 $\lambda(f_1)\overset{e_1}{\longrightarrow }\lambda(f_2)\overset{e_2}{\longrightarrow }\lambda(f_3)$ 
 in $\Gamma^{\chi}$. 
 Since $f_3$ is obtained from $f_2$ by a move which ends at the symbol
 $\times$ with the coordinate $b'(e_2)$,
the position with this coordinate in $f_2$ is empty, so $b'(e_1)\not=b'(e_2)$.
Hence $b'$ satisfies (BB). 
It remains to verify that $b$ and $b'$  are decreasingly equivalent.
Set $j:=b(e_1)$, $q:=b'(e_1)$. Then $q$ is the coordinate of the
 $j$th symbol $\times$ in $f_2$ and  $q>0$.
The condition $b(e_1)>b(e_2)$ means that for $i\geq j$ the $i$th symbols $\times$ 
in $f_2$ and $f_3$ have the same coordinates, whereas the condition
$q>b'(e_2)$ means  that for $s\geq q$ one has $f_2(s)=f_3(s)$. Clearly, 
these conditions are equivalent, so  $b$ and $b'$  are decreasingly equivalent.

Consider the remaining case $t=1$. Since $\tau$ preserves 
$\pari$, $\tail$ and $||\ ||_{gr}$, almost all assertions for $t=1$ 
follows from the corresponding assertions for $t=2$. The only exception
is the inequality $\nu<\lambda$ in (ii), which follows from the following observation:
 for an edge $\nu\to\lambda$ in $\Gamma^{\chi}$ for $t=1$, the diagram
of $\lambda$ is obtained 
 from the diagram of $\nu$ either by moving
 symbol(s) $\times$  to the right or by changing the sign 
 from $-$ to $+$. 
\end{proof}

\subsubsection{}
\begin{cor}{propmove}
Take $\lambda,\nu\in\Lambda^{\chi}$ with 
$\Ext^1(L(\lambda),L(\nu))\not=0$. Then

(i) either 
$\lambda$ can be obtained
from $\nu$ by a move of zero degree  or $\nu$ can be obtained
from $\lambda$ by a move of zero degree;

(ii) $\pari(\lambda)\not=\pari(\nu)$.
\end{cor}
\begin{proof}
By~\Cor{corpa}, the graph $\Gamma^{\chi}$ 
satisfies the assumptions of~\Cor{maincor} (iii), which implies the assertions.
\end{proof}

%
%
%
%
%
%
%
%
%


\subsection{Gruson-Serganova character formula}\label{GScharosp}
We retain notation of~\ref{euler};
for $\nu\in\Lambda^{\chi}$ we  introduce the ``Euler character'' $\mathcal{E}_{\nu}$ by~(\ref{eqE}).

\subsubsection{}
A character formula is given by Theorem 4 in~\cite{GS}. Using~\Cor{corpa} we can
write this formula for $\lambda\in\Lambda^{\chi}$
 in the following way:
\begin{equation}\label{GSchar}
\ch L(\lambda)=
\displaystyle\sum_{\nu\in \Lambda^{\chi}} (-1)^{\pari(\lambda)-\pari(\nu)}d^{\lambda,\nu}_< \mathcal{E}_{\nu},\end{equation}
 where $d^{\lambda,\nu}_<$ is the number of increasing paths
from $\nu$ to $\lambda$ in the graph $D^{\chi}$, where $D^{\chi}$ is obtained from $\Gamma^{\chi}$ by doubling the edges $e$ with $\kappa(e)\not=z^d$
(i.e.,  $D^{\chi}=\Gamma^{\chi}$  for $t\not=0$, see~\ref{Dgosp}).

Notice that $d^{\lambda,\nu}_<$ are non-negative integers, 
$d^{\lambda,\lambda}_<=1$ and
$d^{\lambda,\nu}_<\not=0$ implies $\nu\leq \lambda$
and $||\nu||\leq ||\lambda||$ (in particular,
the right-hand side of~(\ref{GSchar}) is finite).

\subsection{}\begin{rem}{genchar}
Recall that each block of atypicality $k$ for $\osp(M|N)$ is equivalent to the principal block
 of $\fosp(2k+t|2k)$ where $t=1$ for odd $M$  and $t=0,2$ for even $M$.
For  a dominant weight $\lambda$ of atypicality $k$ let $\ol{\lambda}$ be the image of $\lambda$ in $\Lambda^{\chi}$
(that is the $\fosp(2k+t|2k)$-module $L(\ol{\lambda})$ is the image of $\osp(M|N)$-module $L(\lambda)$
under the above equivalence). 
It turns out that this equivalence ``preserve tails'', i.e.
$\tail(\lambda)=\tail(\ol{\lambda})$. 

Introducing
$\pari(\lambda):=\pari(\ol{\lambda})$ we obtain
$\Ext^1(L(\lambda),L(\nu))=0$ if
$\pari(\lambda)=\pari(\nu)$. 
By~\cite{GS}, the formula~(\ref{GSchar}) holds for an arbitrary dominant 
weight $\lambda$ if we introduce
$\mathcal{E}_{\nu}$
by the formula~(\ref{eqE}) and set $d^{\lambda,\nu}_<:=d^{\ol{\lambda},\ol{\nu}}_<$. 
\end{rem}

\section{Arch diagrams}\label{arcs}
In the cases when $\fg$ is not exceptional and the flag of parabolic is standard,
 the description  of $\Gamma^{\chi}$ in~\cite{MS},\cite{GS} can be conveniently
presented  in terms of arc diagrams
introduced in~\cite{GSBGG}, \cite{GH}, where the examples
are presented. Below we will present  this description 
of $\Gamma^{\chi}$ for the principal blocks in $\fgl(k|k),\osp(2k+t|2k)$. 

Our diagrams differs from the arc- or cup diagrams of \cite{ES}; we will call these diagrams  ``arch diagram''. 

We take $\fg=\fgl(k|k),\osp(2k+t|2k)$ and
the central character  $\chi$ corresponding to the principal block.
 For $\fg=\fgl(k|k)$ we take 
$\Sigma=\{\vareps_1-\vareps_2,\ldots,\vareps_k-\delta_1,\delta_1-\delta_2,\ldots,\delta_{k-1}-\delta_k\}$ and the flag~(\ref{parabolicchain}) with $\fl^{(i)}\cong\fgl(i|i)$.
For $\osp(2k+t|2k)$ we retain notation of~\ref{hwtpr}.

\subsection{Arch diagram}
Take $\lambda\in\Lambda^{\chi}$. 

For $\osp(2k+t|2k)$ we 
assign to $\lambda$ the weight diagram  as in~\ref{lambda+rho}.
For $\fgl(k|k)$-case  $\lambda+\rho=\sum_{i=1}^k a_i(\varesp_i-\delta_{k+1-i})$ 
and we assign to $\lambda$ a weight diagram
with the symbols $\times$ at the positions $a_1,\ldots,a_k$ and the empty
symbols $\circ$ in other positions.

A {\em generalized  arch diagram} is the following data: 
\begin{itemize}
\item[$\bullet$]
a weight diagram $f$, where
the symbols $\times$ at the zero position are drawn vertically and
 $>$ (if it is present) is drawn in the bottom,
\item[$\bullet$]
a collection of non-intersecting arches, where each arch  is 
\begin{itemize}
\item
either $\arc(a;b)$  connecting
 the symbol $\times$  with the empty symbol
at the position $b$;
\item
or $\arc(0;b,b')$  connecting
 the symbol $\times$ at the zero position with two empty symbols
at the positions $b<b'$;
\end{itemize}

\end{itemize}

An empty position  is called {\em free} 
 if this position is not an end of an arch.

We call $\arc(a;b)$ a {\em two-legged arch supported at} $a$ and 
$\arc(0;b,b')$ a {\em three-legged arch supported at} $0$.

A generalized arch diagram is called {\em arch diagram} if  

\begin{itemize}
\item[$\bullet$]
each symbol $\times$ is the left end of exactly one arch;

\item[$\bullet$]
there are no free positions under the arches;

\item[$\bullet$]
for the $\fgl$-case all arches are two-legged;

\item[$\bullet$]
for the $\osp(2k|2k),\osp(2k+1|2k)$-cases the lowest $\times$ at the zero position
supports a two-legged arch and the other symbols $\times$ at the zero position support three-legged arches;

\item[$\bullet$]
for the $\osp(2k+2|2k)$-case  all symbols $\times$  at the zero position support three-legged arches.
\end{itemize}

Each weight diagram $f$ admits a unique arch diagram
which we denote by $\Arc(f)$; this diagram can be constructed
in the following way:
we pass from right to left through the weight diagram and connect each symbol $\times$ with the next empty symbol(s) to the right by an arch.

\subsubsection{Partial order}
We consider a partial order on the set of arches by saying that one arch is smaller than
another one if the first one is "below" the second one; one has
$$\begin{array}{l}
arc(a;b)>arc(a';b')\ \ \Longleftrightarrow\ \ a<a'<b\\
arc(0;b_1,b_2)>arc(a';b')\ \ \Longleftrightarrow\ \ a'<b_2,
\end{array}$$
in addition, any two distinct three-legged arches 
are comparable.

\subsubsection{}
For a weight diagram $f$ we  denote by $l_f(p,q)$ the number of $\times-$ the number of $\circ$
strictly between the positions $p$ and $q$. We  denote by  $(f)_p^q$ the diagram
which obtained from $f$ by moving  $\times$ from the position $p$ to a free position $q>p$; such diagram is defined only if
$$ f(p)\in \{\times^i,\overset{\ \ \times^{i}}{>}\} \text{ for }s\geq 1,\ 
f(q)=\circ.$$ 
For instance, for $f=\times^2\circ\times$ one has $(f)_0^3=\times\circ\times\times$ and $(f)_0^2, (f)_1^5$ are not defined.
 If $f(0)=\times^i$
or $\overset{\ \ \times^{i}}{>}$ for $i>1$, we denote by $(f)_{0,0}^{p,q}$ 
the diagram
which obtained from $f$ by moving two symbols $\times$ from the zero position 
 to  free positions $p$ and $q$ with $p<q$; for example, $(\times^2\times)_{0,0}^{3,4}=\circ\times\circ\times\times$.

\subsection{Description of $\Gamma^{\chi}$ in diagrammatic terms} 
Below we present the description for the cases $\fgl(k|k), \osp(2k|2k)$ and $\osp(2k+2|2k)$ (the map $\tau$ gives a bijection between the graphs for
$\osp(2k+1|2k)$ and $\osp(2k+2|2k)$).

Take $\nu,\lambda\in\Lambda^{\chi}$ and denote by $f$ and $g$ their weight diagrams.  
We denote by $l_f(a,b)$ the number of $\times$ minus the number of $\circ$
with the coordinates strictly between $a$ and $b$. Note that
$l_f(a,b)=l_{f'}(a,b)$ if $f'=(f)_a^b$.

The graph  $\Gamma^{\chi}$ contains an edge 
$\nu\overset{e}{\longrightarrow} \lambda$ if and only if $g=(f)_a^p$
or $g=(f)^{p,q}_{0,0}$ and  the folowing conditions hold.

\subsubsection{Case  $g=(f)_a^p$ and $f(a)=\times$} \label{521}
In this case $\Arc(f)$ contains a 
two-legged arch $\arc(a;a_-)$ with $a<p\leq a_-$;
 one has $\kappa(e)=z^{l_f(a,p)}$ and
$l_f(a,b)=l_g(a,b)$.

\subsubsection{Case  $g=(f)_0^p$ and $f(0)\not=\times$}
In this case  $\Arc(f)$  contains three-legged arches (i.e., either $\fg=\osp(2k|2k)$ and
$tail(f)\geq 2$ or $\fg=\osp(2k+2|2k)$ and $\tail(f)\geq 1$).
For this case
we denote by $\arc(0;a',a_+)$  the highest three-legged
arch in $\Arc(f)$ (for example,
for $\nu=0$ this is $\arc(0;2k-2,2k-1)$ if $\fg=\osp(2k|2k), k>1$
and $\arc(0;2k-1,2k)$ if $\fg=\osp(2k+2|2k)$).
In this case $p\leq a_+$. 

If $f(0)=\overset{\ \ \times^{i+1}}{>}$ for $i\geq 0$ (then $t=2$) we have
 $\kappa(e)=z^{2i+1+l_f(0,p)}$. In the remainng case
 $f(0)=\times^{i+1}$, $i>0$ (then $t=0$) there is a unique
two-legged 
$\arc(0;a_-)$    supported at the zero position and
$\kappa(e)=z^{2i+l_f(0,p)}$ if $p>a_-$ and 
$\kappa(e)=z^{2i+l_f(0,p)}+z^{l_f(0,p)}$ if $p\leq a_-$.

For example, for $f=\times^3\times\circ\circ\circ\times$
and $g=(f)_0^2=\times^2\times\times\circ\circ\times$ one has $\kappa(e)=z+z^5$
and for $g'=(f)_0^4=\times^2\times\circ\circ\times\times$ one has $\kappa(e')=z^3$.

\subsubsection{Case $g=(f)^{p,q}_{0,0}$ }\label{524}
In this case $\Arc(f)$ contains a three-legged arch  $\arc(0;p,a_2)$
supported at the zero position which is not the highest arch 
and  $p<q\leq a_2$; one has 
$\kappa(e)=z^{l_f(p,q)}$.

Examples.

For $t=0$, $\tail(f)\leq 2$ (resp., $t=2$, $\tail(f)\leq 1$)
there are no suitable three-legged arches.

For $f=\times^3\circ\circ\times\times$ there is only one
suitable three-legged arch, which is $\arc(0;2,7)$. Thus $p=2$, 
$q\in\{5,6,7\}$ and $\kappa(e)=z^{7-q}$.

For $f=\overset{\ \ \times^{2}}{>}\circ\times\times\circ\times$  there is only one
suitable three-legged arch, which is $\arc(0;1,8)$. Thus $p=1$,
 $q\in\{4,6, 7,8 \}$ with $\kappa(e)=2$ for $q=5,6$ and
$\kappa(e)=8-q$ for $q=7,8$.

\subsection{Applications to $\Ext$-graph}
Take $\nu,\lambda\in\Lambda^{\chi}$ with $\nu<\lambda$ and denote by $f$ and $g$ their weight diagrams.

By~\cite{MS}, in the $\fgl$-case $\Ext^1(L(\lambda),L(\nu))\not=0$  if and only if 
$g=(f)_a^p$, where $\arc(a,p)\in \Arc(f)$ (i.e. $g$ can be obtained from $f$ by moving a symbol $\times$ along the arch 
supported by this symbol).

By above, for $\fg=\osp(2k+t|2k)$ with $t=0,2$  if $\Ext^1(L(\lambda),L(\nu))\not=0$, then either
$g=(f)_a^p$, where $\arc(a,p)\in \Arc(f)$ or
$g=(f)_0^{a_+}$ where $\arc(0;a',a_+)$ is the highest arch supported at $0$, or
$g=(f)^{a_1,a_2}_{0,0}$, where $\arc(0;a_1,a_2)$ is not the highest arch supported at $0$.

\section{Appendix: useful facts about $\Ext^1$}\label{sectapp}

\subsection{}
We start from the following lemma which express 
$\dim \Ext^1(L',L)$ in terms of indecomposable extensions.

\begin{lem}{lemExt}
Let $A$ be an associative algebra and $L,L'$ be simple non-isomorphic modules
over $A$ with $\End_A(L)=\End_A(L')=\mathbb{C}$.
Let $N$ be a module with
\begin{equation}\label{Nini}
\Soc N=L^{\oplus m},\ \ N/\Soc N=L'.\end{equation}

(i) If $N$ is indecomposable, then $m\leq \dim \Ext^1(L',L)$.

(ii) If $m\leq \dim \Ext^1(L',L)$, then there exists an indecomposable $N$ satisfying~(\ref{Nini}).
\end{lem}
\begin{proof}
Consider any exact sequence of the form
\begin{equation}\label{ExtN}
0\to L^{\oplus m}\overset{\iota}{\longrightarrow} N
\overset{\phi}\longrightarrow L'\to 0.\end{equation}
For each $i=1,\ldots,m$ let
$p_i$ be the projection from $L^{\oplus m}$ to the $i$th component
and $\theta_i$ is the corresponding embedding $p_i\theta_i=Id_L$.
Consider a commutative diagram 
\begin{equation}\label{commdia}
\xymatrix{& 0\ar[r]\ar[d] & L^{\oplus m} \ar^{\iota}[r]\ar^{p_i}[d]&
N\ar^{\phi}[r]\ar^{\psi_i}[d]&L'\ar[r]\ar^{Id}[d]& 0\ar[d]&\\
& 0\ar[r] & L\ar^{\theta_i}[u]
 \ar^{\iota_i}[r]&
M^i\ar^{\phi_i}[r]&L'\ar[r]& 0&
}
\end{equation}
where $\psi_i: N\to M^i$ is a surjective map with 
$\Ker\psi_i=\Ker p_i$. 

The bottom line
is an element of $\Ext^1(L',L)$, which we denote by $\Phi_i$.
If $m>\dim \Ext^1(L',L)$, then $\{\Phi_i\}_{i=1}^m$ are linearly dependent 
and we can assume that $\Phi_1=0$, so $\Phi_1$ splits. Let
 $\tilde{p}: M^1\to L$ be the projection, i.e.
$\iota_1\tilde{p}=Id_L$. Consider the maps
$$L\overset{\iota\circ\theta_1}{\longrightarrow} N \overset{\tilde{p}\circ\psi_1}{\longrightarrow}L.$$
The composed map
$$\tilde{p}\circ\psi_1\circ\iota\circ\theta_1: L\to L$$
is surjective, so it is an isomorphism. 
Hence $N$ is decomposable. This establishes
(i).

For (ii)  let $\{\Phi_i\}_{i=1}^m$ be linearly independent  elements in  $\Ext^1(L',L)$, i.e.
$$
\Phi_i:\ \ \ \ \ 0\to L\overset{\iota_i}{\longrightarrow} M^i\overset{\phi_i}{\longrightarrow} L'\to 0.$$
Consider the exact sequence 
$$0\to L^{\oplus m}\longrightarrow  \oplus M_i\to (L')^{\oplus m}\to 0.$$
Let $diag(L')$ be the diagonal copy of $L'$ in $(L')^{\oplus m}$ and
let $N$ be the preimage of $diag(L')$
in $\oplus M^i$. This gives the exact sequence of the form~(\ref{ExtN})
and the commutative diagram~(\ref{commdia}).
Let us show that $N$ is indecomposable. Assume that $N$ decomposable, so
 $N=N_1\oplus N_2$. Since $L'\not\cong L$ one has
$\phi(N_1)=0$ or $\phi(N_2)=0$, so $N_1$ or $N_2$  lies in the socle of $N$.
Therefore $N$ can be written as $N=L^1\oplus N''$ with $L^1\cong L$. 
Since $\dim \Hom(L,N)=m$ one has
$\dim \Hom(L,N'')=m-1$.
Changing the basis in the span of $\{\Phi_i\}_{i=1}^m$,
we can assume that 
$\Ker p_1\subset N''$. Since
$\Ker\psi_1=\Ker p_1$, the exact sequence $\Phi_1$ splits, a contradiction.
Hence $N$ is indecomposable. 
This completes the proof.
\end{proof}

\subsection{Notation}
Let $\fg$ be a Lie superalgebra of at most countable dimension. 

Let $\fh\subset\fg_0$ be a finite-dimensional subalgebra
satisfying


(H1) $\fh$ acts diagonally on $\fg$ and $\fg_0^{\fh}=\fh$.

We choose $h\in\fh$ satisfying

(H2)   $\fg^{h}=\fg^{\fh}$ and each non-zero eigenvalue
 of $\ad h$ has a non-zero real part.

(The assumption on $\dim\fg$ ensures the existence of $h$).

We write $\fg=\fg^{\fh}\oplus (\oplus_{\alpha\in\Delta(\fg)}\fg_{\alpha})$,
with  $\Delta(\fg)\subset\fh^*$ and
$$\begin{array}{l}

 \fg_{\alpha}:=\{g\in\fg|\ [h,g]=\alpha(h)g\ \text{ for all }h\in\fh\}.
 \end{array}$$
We introduce the triangular decomposition 
$\Delta(\fg)=\Delta^+(\fg)\coprod \Delta^-(\fg)$,
with
$$\Delta^{\pm}(\fg):=\{\alpha\in\Delta(\fg)|\ \pm Re\, \alpha(h)>0\},$$
 and  define the  partial order on $\fh^*$  by
$\lambda>\nu$  if $\nu-\lambda\in\mathbb{N}\Delta^-$. 
We set 
$\fn^{\pm}:=\oplus_{\alpha\in\Delta^{\pm}}\fg_{\alpha}$ and consider 
the Borel subalgebra $\fb:=\fg^{\fh}\oplus\fn^+$.

\subsubsection{}\label{patpat}
Take  $z\in\fh$ satisfying
\begin{equation}\label{condz}
\alpha(z)\in\mathbb{R}_{\geq 0}\ \text{ for  }\alpha\in\Delta^+\ \text{ and }
\alpha(z)\in\mathbb{R}_{\leq 0}\ \text{ for }\alpha\in\Delta^-.
\end{equation}
Consider the superalgebras
$$\ft:=\ft(z):=\fg^z,\ \ \ \fp:=\fp(z):=\fg^z+\fb.$$
 Notice that $\fp=\ft\rtimes\fm$,
 where 
 $\fm:= \oplus_{\alpha\in\Delta: \alpha(z)>0}\fg_{\alpha}$. 
Both triples $(\fp(z),\fh,h)$, $(\ft(z),\fh,h)$ satisfy (H1), (H2). One
has $\ft^{\fh}=\fp^{\fh}=\fg^{\fh}$ and
$$\begin{array}{l}
 \ \Delta^+(\fp)=\Delta^+(\fg),\ \ \ \ 
\ \Delta^+(\ft)=\{\alpha\in\Delta^{+}(\fg)|\ \alpha(z)=0\}\\
\Delta^-(\fp)=\Delta^-(\ft)=\{\alpha\in\Delta^{-}(\fg)|\ \alpha(z)=0\}
\end{array}$$

\subsubsection{Modules $M(\lambda),L(\lambda)$}\label{Mlambda}
For  a semisimple  $\fh$-module $N$ we
 denote by $N_{\nu}$ the weight space of the weight $\nu$ and by
 $\Omega(N)$ the set of weights of $N$.
 
 We denote by $\CO$ the full category of 
finitely generated modules with a diagonal action of $\fh$
and locally nilpotent action of $\fn$.

By Dixmier generalization of Schur's Lemma (see~\cite{Dix}),
up to a parity change, the simple $\fg^{\fh}$-modules are 
parametrized by
$\lambda\in\fh^*$; we denote by $C_{\lambda}$ a simple
$\fg^{\fh}$-module, where $\fh$ acts by $\lambda$.
We view $C_{\lambda}$ as a $\fb$-module with the zero action of $\fn$
and set
$$M(\lambda):=\Ind^{\fg}_{\fb} C_{\lambda};$$  
 this module
has a  unique simple quotient which we denote by 
$L(\lambda)$. 
(The module $M(\lambda)$ is a Verma module
if $\fg^{\fh}=0$). 
We introduce similarly the modules
$M_{\fp}(\lambda)$, 
 $L_{\fp}(\lambda)$ for the algebra $\fp$ and 
 $M_{\ft}(\lambda)$, $L_{\ft}(\lambda)$ for the algebra $\ft$.

\subsubsection{Set $N(\fg;m)$}
Let $\lambda\not=\nu\in\fh^*$ be such that  $Re (\lambda-\nu)(h)\geq 0$.
If 
$$0\to L(\nu)\to E\to L(\lambda)\to 0$$
is a non-split exact sequence, then $E$ is generated
by $E_{\lambda}\cong C_{\lambda}$, so
$E$ is a quotient of $M(\lambda)$ and $\nu<\lambda$.

For $\lambda,\nu\in\fh^*$ 
we denote by $\cN(\fg;m)$ the set of indecomposable $\fg$-modules
$N$ such that 
\begin{equation}\label{Nm}
\Soc N=L(\nu)^{\oplus m},\ \ \  N/\Soc N=L(\lambda)\end{equation}
 By~\Lem{lemExt} one has
\begin{equation}\label{ExtNm}
\dim\Ext^1(L(\lambda),L(\nu))=\max\{m|\ \cN(\fg;m)\not=0\}.
\end{equation}

 Note that each module $N\in\cN (\fg;m)$ is a quotient
of $M(\lambda)=\Ind^{\fg}_{\fb} L_{\fb}(\lambda)$.
We set
\begin{equation}\label{mgp}
m(\fg;\fp;\lambda;\nu):=max\{m|\ \exists 
 N\in \cN(\fg;m)\ \text{ which is a quotient of } \Ind^{\fg}_{\fp} L_{\fp}(\lambda)\}.
 \end{equation}

\subsubsection{}
\begin{cor}{corExtM}
Take $\lambda\not=\nu\in\fh^*$ with $Re (\lambda-\nu)(h)\geq 0$.

(i) If $\Ext^1(L(\lambda),L(\nu))\not=0$, then 
$\lambda>\nu$;

(ii) $\Ext^1(L(\lambda),L(\nu))=\Ext^1_{\CO}(L(\lambda),L(\nu))$;

(iii)
$\dim\Ext^1(L(\lambda),L(\nu))=m(\fg;\fb;\lambda;\nu)\leq \dim M(\lambda)_{\nu}$.
\end{cor}

\subsection{Remark}
\label{remKM}
If $\fg$ is a Kac-Moody superalgebra, then $\fg$ admits antiautomorphism which stabilizes the elements of
$\fh$ and the category $\CO(\fg)$ admits
 a duality functor $\#$
with the property $L^{\#}\cong L$ for each simple module
$L\in\CO(\fg)$. In this case
$$\dim \Ext^1(L(\lambda),L(\nu))=\dim \Ext^1(L(\nu),L(\lambda)).$$

\subsection{}
The following lemma is a slight reformulation of Lemma 6.3 in~\cite{MS}.

\begin{lem}{lemExtVerma}
Take $\lambda,\nu\in\fh^*$ with $\lambda>\nu$. 

(i) $m(\fg;\fb;\lambda;\nu)\leq m(\fp;\fb;\lambda;\nu)$ if $\nu-\lambda\in \mathbb{N}\Delta^-(\fp)$;

(ii) 
$m(\fg;\fb;\lambda;\nu)=
m(\fg;\fp;\lambda;\nu)$  if $\nu-\lambda\not\in \mathbb{N}\Delta^-(\fp)$.
\end{lem}
\begin{proof}
Write $\fp=\fg^z\rtimes\fm$ as in~\ref{patpat}. 
For each $\fg$-module $M$ set
$$\Res(M):=\{v\in M| \ zv=\lambda(z)v\}.$$
Note that $\Res(M)$ is a $\fg^z$-module; we view $\Res(M)$ as a $\fp$-module 
with the zero action of $\fm$. This defines an exact functor
$\Res: \fg-Mod\to \fp-Mod$. By the PBW Theorem
$$
\Res(M(\mu))=\left\{\begin{array}{ll}
M_{\fp}(\mu) &\ \text{ if } \mu(z)=\lambda(z)\\
0&\ \text{ if }  (\lambda-\mu)(z)<0\end{array}\right.
$$
Let us show that 
\begin{equation}\label{Res}
\Res(L(\mu))=\left\{\begin{array}{ll}
L_{\fp}(\mu) &\ \text{ if }  \mu(z)=\lambda(z)\\
0&\ \text{ if } (\lambda-\mu)(z)<0.\end{array}\right.\end{equation}
Indeed, since $\Res$ is exact, 
$\Res(L(\mu))$ is a quotient of $\Res(M(\mu))$; this gives the second formula.
For the first formula assume that $\mu(z)=\lambda(z)$ and that $E$ is a
proper submodule of 
$\Res(L(\mu))$. Since $E$ is a $\fp$-module we have
$\cU(\fg)E=\cU(\fn^-)E$. 
Since $\Res(L(\mu))$ is a quotient of $\Res(M(\mu))$
and
$$(\Res(M(\mu)))_{\mu}=(M_{\fp}(\mu))_{\mu}$$
is a simple $\fg^{\fh}$-module, one has $\gamma<\mu$ for each $\gamma\in\Omega(E)$. Therefore $(\cU(\fn^-)E)_{\mu}=0$, so
$\cU(\fg)E$ is a proper $\fg$-submodule of $L(\mu)$. Hence $E=0$, so
 $\Res(L(\mu))$ is simple. This  establishes~(\ref{Res}).

Now we fix a non-negative integer $m\leq  m(\fg;\fb;\lambda;\nu)$ and 
$N\in \cN(\fg;m)$.

Consider the case when $\nu-\lambda\in \mathbb{N}\Delta^-(\fp)$.
Then $\lambda(z)=\nu(z)$, so
 $\Res(L(\nu))=L_{\fp}(\nu)$.
Since $\Res$ is exact one has
$$\Soc (\Res N)=L_{\fp}(\nu)^{\oplus m},\ \ \ 
\Res(N)/\Soc (\Res (N))=L_{\fp}(\lambda)$$
and $\Res(N)$ is a quotient of $M_{\fp}(\lambda)$.
Therefore $m\leq m(\fp;\fb;\lambda;\nu)$. This gives (i).

For (ii)  $\nu-\lambda\not\in \mathbb{N}\Delta^-(\fp)$.
 Let us show  that $N$ is a quotient of $\Ind_{\fp}^{\fg} L_{\fp}(\lambda)$. 
Write
$$\Ind_{\fp}^{\fg} L_{\fp}(\lambda)=M(\lambda)/J,\ \ \ L_{\fp}(\lambda)=M_{\fp}(\lambda)/J'$$
where $J$ (resp., $J'$) is the corresponding submodule of $M(\lambda)$
(resp., of  $M_{\fp}(\lambda)$).
Since $\Ind^{\fg}_{\fp}$ is exact and $\Ind^{\fg}_{\fp} M_{\fp}(\lambda)=M(\lambda)$
one has
$J\cong \Ind^{\fg}_{\fp} J'$; in particular, 
 each maximal element in $\Omega(J)$ lies in $\Omega(J')$.
Note that 
$$\Omega(J')\subset \lambda-\mathbb{N}\Delta^+(\fl).$$

Let $\phi: M(\lambda)\twoheadrightarrow N$ be the canonical surjection.
Since $J_{\lambda}=0$, $\phi(J)$ is a proper submodule
of $N$, so  $\phi(J)$ is a submodule of
$\Soc(N)=L(\nu)^{\oplus m}$. 

If $\phi(J)\not=0$, then $\nu$ is a maximal element
in $\Omega(J)$ and so $\lambda-\nu\in \mathbb{N}\Delta^+(\fl)$,
which contradicts to $\nu-\lambda\in\mathbb{N}\Delta^-(\fp)$.
Therefore $\phi$ induces a map 
$\Ind_{\fp}^{\fg} L_{\fp}(\lambda)=M(\lambda)/J \twoheadrightarrow N$.
Therefore 
$$m(\fg;\fb;\lambda;\nu)\leq m(\fg;\fp;\lambda;\nu).$$
Since $\Ind_{\fp}^{\fg} L_{\fp}(\lambda)$ is a quotient of $M(\lambda)$,
we have 
$m(\fg;\fb;\lambda;\nu)\geq m(\fg;\fp;\lambda;\nu)$.
Thus $m(\fg;\fb;\lambda;\nu)=m(\fg;\fp;\lambda;\nu)$ as required.
\end{proof}

\subsection{}
\begin{cor}{corExtchain}
Let $z_1,\ldots,z_{k-1}\in\fh$ satisfying~(\ref{condz}) be such that
 $\ft^{z_i}\subset \ft^{z_{i+1}}$. Set $\ft^{(i)}:=\ft^{z_i}$
and consider the chain
 $$\fh=:\ft^{(0)}\subset\fp^{(1)}\subset \fp^{(2)}\subset\ldots\subset \ft^{(k)}:=\fg.$$
For $\lambda\not=\nu\in\fh^*$ with $\lambda>\nu$
one has
$$\Ext^1(L(\lambda),L(\nu))\leq m(\ft^{(s)};\fp;\lambda;\nu),$$
where $s$ is minimal such that $\nu-\lambda\in\mathbb{N}\Delta^-(\ft^{(s)})$ and $\fp:=(\ft^{(s-1)}+\fb)\cap \ft^{(s)}$.
\end{cor}
\begin{proof} 
Take $\fp^{(i)}:=\ft^{(i)}+\fb$.
Combining~\Cor{corExtM} and~\Lem{lemExtVerma} we obtain
$$\Ext^1(L(\lambda),L(\nu))=m(\fg;\fb;\lambda;\nu)\leq m(\fp^{(s)};\fp^{(s-1)};\lambda;\nu).$$
Recall that
$\fp^{(i)}=\ft^{(i)}\rtimes\fm^{(i)}$, where $\fm^{(i)}$ is the maximal  ideal in
$\fp^{(i)}$ which lie in $\fn$. 
One has $\fm^{(s)}\subset\fm^{(s-1)}$. In particular,
$\fm^{(s)}$ annihilates $\Ind^{\fp^{(s)}}_{\fp^{(s-1)}} L_{\fp^{(s-1)}}(\lambda)$,
so
$$m(\fp^{(s)};\fp^{(s-1)};\lambda;\nu)=m(\ft^{(s)};\fp^{(s-1)}/\fm^{(s)};\lambda;\nu).
$$
Since 
$$\fp^{(s-1)}=\ft^{(s-1)}+\fb\subset\fp^{(s)}=\ft^{(s)}\rtimes\fm^{(s)}$$
the image of $\fp^{(s-1)}$ in
$\ft^{(s)}=\fp^{(s)}/\fm^{(s)}$ 
coincides with $\fp\subset \ft^{(s)}$.
\end{proof}


\subsubsection{Remark}\label{remh0}
Consider the special case when $\fg^{\fh}=\fh$ and $\ft^{(s)}=\fl'\times \fh''$, where $\fh''\subset\fh$. Take 
$$\fh':=\fl\cap \fh,\ \ \ \fb':=\fl\cap\fb,\ \ \ \fp':=\fl'\cap \fp.$$
 and let $h'$ be  the image of
$z_s$ in $\fh'$. 
Then the triple $(\fl';\fh';h')$ satisfies (H1), (H2)
and $\fb'$ is the Borel subalgebra of $\fl'$. For each 
 $\lambda\in\fh^*$ we set 
$$\lambda':=\lambda|_{\fh'}.$$
Since $\nu-\lambda\in\mathbb{N}\Delta^-(\ft^{(s)})$ one has
$$m(\ft^{(s)};\fp;\lambda;\nu)=m(\fl', \fp',\lambda',\nu').$$

Assume, in addition, that $L_{\fl'}(\lambda'), L_{\fl'}(\nu')$ 
are finite-dimensional. It is not hard to see that $\Ind^{\fl'}_{\fp'} (L_{\fp'}(\lambda'))$  admits a unique maximal finite-dimensional subquotient which we denote by  $\Gamma_{\fl',\fp'}(L_{\fp'}(\lambda'))$ and that for any  finite-dimensional quotient $N'$ of $\Ind^{\fl'}_{\fp'} (L_{\fp'}(\lambda'))$ 
there exists an epimorphism $\Gamma_{\fl',\fp'}(L_{\fp'}(\lambda'))\to N'$.
This implies
$$\Ext^1(L(\lambda),L(\nu))\leq m(\fl', \fp',\lambda',\nu')\leq 
[\Gamma_{\fl',\fp'}(L_{\fp'}(\lambda')):L_{\fl'}(\nu')].$$

\end{document}